\documentclass{amsart}[11pt,letter]
\usepackage[margin=2.54cm]{geometry}
\usepackage{amsmath,amssymb,amsthm,url,tikz,chngcntr}
\usetikzlibrary{positioning}
\counterwithin{table}{section}
\counterwithin{figure}{section}
\usepackage{hyperref}

\newtheorem{theorem}{Theorem}[section]
\newtheorem{lemma}[theorem]{Lemma}

\numberwithin{equation}{section}

\newtheorem{problem}[theorem]{Problem}

\newtheorem{hypothesis}[theorem]{Hypothesis}
\newtheorem{remark}[theorem]{Remark}

\theoremstyle{definition}
\newtheorem{definition}[theorem]{Definition}

\newcommand{\A}{\mathrm A}
\newcommand{\C}{\mathrm C}
\newcommand{\D}{\mathrm D}
\newcommand{\DW}{\mathrm DW}
\newcommand{\K}{\mathrm K}
\newcommand{\rH}{\mathrm H}
\newcommand{\V}{\mathrm V}
\newcommand{\E}{\mathrm E}
\newcommand{\cC}{{\mathcal{C}}}

\newcommand{\NN}{\mathbb{N}}

\newcommand{\vC}{\vec{\C}}
\newcommand{\vGa}{\vec{\Gamma}}

\renewcommand{\wr}{\mathop{\rm wr}}

\newcommand{\ZZ}{\mathbb{Z}}

\newcommand{\Prism}{\mathrm{Pr}}
\newcommand{\Mob}{\mathrm{Mb}}

\newcommand{\Spl}{\mathrm S}

\newcommand{\Fix}{\mathrm{Fix}\,}

\newcommand{\Cay}{\mathrm{Cay}}

\newcommand{\Aut}{\mathrm{Aut}}

\newcommand{\Alt}{\mathrm{Alt}}
\newcommand{\Sym}{\mathrm{Sym}}

\newcommand{\PSU}{\operatorname{\mathrm{PSU}}}
\newcommand{\PSp}{\mathop{\mathrm{PSp}}\nolimits}
\newcommand{\POmega}{\mathop{\mathrm{P}\Omega}\nolimits}

\newcommand{\PSL}{\mathrm{PSL}}
\newcommand{\PGL}{\mathrm{PGL}}

\newcommand{\fpr}{\mathrm{fpr}}
\newcommand{\lcm}{\mathrm{lcm}}

\newcommand{\cM}{\mathcal{M}}
\newcommand{\cF}{\mathcal{F}}
\def\cent#1#2{{\bf C}_{{#1}}({{#2}})}
\def\Z#1{{\bf Z}{{(#1)}}}
\begin{document}
\title[Fixed points of automorphisms of graphs]{
On the number of fixed points of automorphisms of vertex-transitive graphs
}

\author{Primo\v{z} Poto\v{c}nik}
\address{Faculty of Mathematics and Physics, University of Ljubljana, Jadranska 21, SI-1000 Ljubljana, Slovenia\\
 also affiliated with 
Institute of Mathematics, Physics and
  Mechanics, Jadranska 19, SI-1000 Ljubljana, Slovenia}
 \email{primoz.potocnik@fmf.uni-lj.si}

\author{Pablo Spiga}
\address{Dipartimento di Matematica e Applicazioni, University of Milano-Bicocca, Via Cozzi 55, 20125 Milano, Italy} 
\email{pablo.spiga@unimib.it}
\begin{abstract}
The main result of this paper is that, if $\Gamma$ is a finite connected
$4$-valent vertex- and edge-transitive graph, then
either $\Gamma$ is part of a well-understood family of graphs,
or every non-identity automorphism of $\Gamma$ fixes at most $1/3$ of the vertices. As a corollary, we get a similar result for
$3$-valent vertex-transitive graphs. 
\end{abstract}

\thanks{The first-named author gratefully acknowledges the support of the Slovenian Research Agency ARRS, core funding programme P1-0294 and research project J1-1691.}

\maketitle

\section{Introduction}\label{sec:intro}

The aim of this paper is to study a graph-theoretical parameter called {\em fixicity},
 defined as the maximal number of vertices that are fixed by a non-trivial automorphism
of the graph. Investigation of a group theoretical analogue of this parameter (the maximum number of points fixed by a permutation group) has a long history
going back to a classical work of Jordan studying primitive permutation group
containing a non-trivial permutation fixing all but a prescribed number of points.
His results were later improved significantly by several authors: for example, Babai \cite{Bab},
Liebeck and Saxl \cite{LS}, Guralnick and Magaard \cite{GurMag}, Burnes \cite{Tim1,Tim4}, Liebeck and Shalev \cite{LS}, to name a few.
 As a result, all primitive groups $G$ having a non-trivial permutation fixing more than half of the points are known.

To the best of our knowledge, the fixicity of a graph was first studied by Babai 
  \cite{Babai1,Babai2} and was motivated by the famous graph isomorphism problem \cite{Babai3}.
   In these papers,
  Babai shows how fixicity is related to a number of important notions, such as the spectrum of the graph, the order of individual automorphisms and the automorphism group of the graph. While the focus there are strongly regular graphs (that can be though of as graphs are highly symmetrical through from a purely combinatorial point of view)
this paper is devoted to the fixicity of graphs exhibiting a high level of symmetry as measured
through their automorphism groups. In particular, we will be interested in connected graphs of valence at most $4$ admitting a group of automorphisms $G$ acting transitively on the vertices ($G$-vertex-transitive graphs),
edges ($G$-edge-transitive graphs) and/or ordered pairs of adjacent vertices 
($G$-arc-transitive graphs).

Our understanding of vertex-transitive graphs 
 is a function of time. The fact that this function has increased so much recently (especially for graphs of valency $3$ and $4$) is, in our opinion, due to two processes intimately intertwined. On the one hand, theoretical results allow us to get deeper into the structure (both combinatorial and algebraic) of vertex-transitive graphs. These results can often be used to improve our  database of vertex-transitive graphs, see~\cite{condercensus,conder,census1,census2,census3}. On the other hand, these databases can be used to test open problems or to formulate conjectures; see for example \cite{ConVer,PotVid,AlePri}. The  spin off of this process is more theoretical work. And the loop starts again, if one can really say that there is a ``start" and an ``end" in this process.

The pattern described in this paper starts with some computer evidence, found by Gabriel Verret and the first-named author of this paper. By checking the census of connected $3$-valent vertex-transitive  graphs~\cite{census1,census2} (which was obtained from the theoretical work in \cite{PSV}), they observed that (for graphs small enough to be in this list) non-identity automorphisms of a connected $3$-valent vertex-transitive graph $\Gamma$ cannot fix more than $1/3$ of the  vertices of $\Gamma$, unless $\Gamma$ is in a very special family or very small. A similar pattern holds for the family of connected $4$-valent vertex- and edge-transitive graphs.

Our main results are the following. For not breaking the flow of the argument, we refer the reader to Sections~\ref{sub1} and~\ref{sub2} for undefined terminology, including the definition of the Praeger-Xu graphs $\C(r,s)$ and the Split Praeger-Xu graphs $\Spl(\C(r,s))$.

\begin{theorem}
\label{thrm:1}
Let $\Gamma$ be a finite connected edge- and vertex-transitive $4$-valent graph
admitting a non-identity automorphism fixing more than $1/3$ of the vertices. 
Then $\Gamma$ is arc-transitive and one of the following holds:
\begin{description}
\item[(i)] $|\V\Gamma| \le 70$ and $\Gamma$ is one of the six exceptions $\Psi_1, \ldots, \Psi_6$, defined in Section~$\ref{sub1}$;
\item[(ii)]$\Gamma$ is isomorphic to a Praeger-Xu graph $\C(r,s)$ with $1\le s< 2r/3$ and $r\ge 3$.
\end{description} 
\end{theorem}

\begin{theorem}
\label{thrm:2}
Let $\Gamma$ be a finite connected $3$-valent vertex-transitive graph admitting a non-identity automorphism fixing more than $1/3$ of the vertices. Then one of the following holds:
\begin{description}
\item[(i)] $|\V\Gamma| \le 20$ and $\Gamma$ is one of the six exceptions $\Lambda_1, \ldots, \Lambda_6$, defined in Section~$\ref{sub1}$;
\item[(ii)]$\Gamma$ is isomorphic to a Split Praeger-Xu graph $\Spl(\C(r,s))$  with $1\le s< 2r/3$ and $r\ge 3$.
\end{description} 
\end{theorem}

Observe that every primitive permutation group $G\le \Sym(\Omega)$ acts as an edge- and vertex-transitive group of automorphisms on each of its non-trivial orbital graphs on  $\Omega$.
The results of this paper can thus be considered as an attempt to generalise the theorems about the maximal number of fixed points of a non-identity element in a primitive permutation group; see for example \cite{GurMag,La,LS}.
One of the key ingredients in our proof of Theorems~\ref{thrm:1} is the following recent result \cite{PotSpiO2}. 
Since its proof  depends heavily on the classification of finite simple groups,
 so do the proofs of Theorems~\ref{thrm:1} and~\ref{thrm:2}.

\begin{theorem}
\cite[Theorem 1.1]{PotSpiO2}
\label{proposition:7}
Let $G$ be a transitive permutation group on $\Omega$ containing no non-trivial normal subgroups of order a power of $2$ (that is, ${\bf O}_2(G)=1$) and  let $\omega\in \Omega$ with $G_\omega$ being a $2$-group. Then $|\{\delta\in \Omega\mid \delta^g=\delta\}|\le |\Omega|/3$, for every $g\in G\setminus \{1\}$. 
\end{theorem}
 
The bound $1/3$ in the Theorems~\ref{thrm:1} and ~\ref{thrm:2} 
 is sharp in the sense that there exists an infinite family (in each case) meeting this bound.
To see this, consider the graph $\DW_m$ with vertex-set $\ZZ_m\times \ZZ_3$ and
the edge-set $\{(x,i),(x+1,j) \mid x\in \ZZ_m, i,j\in \ZZ_3, i\not =j\}$. The graph 
$\DW_m$ is clearly connected, $4$-valent and arc-transitive. Moreover, it admits an automorphism
which fixes every vertex of the form $(x,0)$ while swapping the vertices in 
each pair $\{ (x,1), (x,2)\}$, $x\in \ZZ_m$. In a similar way as $4$-valent arc-transitive
Praeger-Xu graphs yield $3$-valent vertex-transitive Split Praeger-Xu graphs (see Section~\ref{sec:subS}), one can apply the splitting ``operation'' to obtain a family of $3$-valent vertex-transitive
graphs $\mathrm{S}(\DW_m)$
of fixity exactly $1/3$ of the number of vertices.
This suggests the following problem:
 
\begin{problem}{\rm Determine the connected $4$-valent arc-transitive graphs and the connected $3$-valent vertex-transitive graphs admitting an automorphism fixing precisely $1/3$ of the vertices.}
\end{problem}

Theorems~\ref{thrm:1} and \ref{thrm:2}
 seem to be suggesting that the proportion of fixed points of a non-identity automorphism of a connected vertex-transitive graph is bounded by a ``small'' constant unless the graph is either small or rather ``special''. 
Since at this point it is not clear to us
what the class of ``special'' graphs of larger valencies might be
and what would be a meaningful ``small constant'',
we include the definition of ``special graphs'' and ``small constant''
as a part of the following problem:

\begin{problem}
\label{prob:openended}{\rm
For a given positive integer $d$ find a ``small constant'' $c_d$ and a ``well-understood''
family of ``special graphs'' ${\mathcal F}_d$ such that every finite connected $d$-valent vertex-transitive graph $\Gamma$ admitting a non-trivial automorphism fixing more than $c_d|\V\Gamma|$ vertices belongs to ${\mathcal F}_d$.}
\end{problem}

Finally, we would like to propose the following ``edge-fixing'' variation of Theorems~\ref{thrm:1} and~\ref{thrm:2}:

\begin{problem}{\rm Determine the connected $4$-valent arc-transitive graphs and the connected $3$-valent vertex-transitive graphs admitting an automorphism fixing more than $1/3$ of the {\em edges}.}
\end{problem}

\subsection{Basic terminology and notation}\label{sub0}

A graph in this paper will be viewed as a pair $(V,E)$ where $V$ is a finite non-empty set of
{\em vertices} and $E$ is a set of unordered pairs of $V$, called {\em edges}. If $\Gamma:=(V,E)$
is a graph, then we let $\V\Gamma:=V$ and $\E\Gamma:= E$. An {\em $s$-arc} of a graph
 is an $(s+1)$-tuple of vertices with every two consecutive vertices adjacent and
every three consecutive vertices pair-wise distinct. In particular, a $1$-arc is also called an {\em arc}. The set of arcs of a graph $\Gamma$ is denoted $\A\Gamma$.

We will also need a notion of a {\em digraph}, which we define to be a pair $(V,A)$,
where $V$ is a finite non-empty set of
{\em vertices} and $A$ is a set of ordered pairs of distinct vertices.
Elements of $A$ are called {\em arcs} of the digraph.
An {\em $s$-arc} of a digraph is an $(s+1)$-tuple of vertices 
such that every two consecutive vertices form an arc.
If $(u,v)$ is an arc of a digraph, then we say that $v$ is an {\em out-neighbour} of $u$ and that
$u$ is an {\em in-neighbour} of $v$. The {\em out-valency} ({\em in-valency}, respectively) of a given vertex is then the number of its in-neighbours (out-neighbours, respectively).
If $\vGa:=(V,A)$ is a digraph, then the underlying graph of $\vGa$ is the graph $(V,E)$
with $E:=\{ \{u,v\} : (u,v) \in A\}$. 
Note that when $\vGa$ is an {\em orientation} (that is, when
$(u,v) \in \A\vGa$ implies $(v,u)\not\in \A\vGa$),
then there is a bijective correspondence between the arcs of $\vGa$
and edges of the underlying graph.

Let $\Gamma$ be a graph (or a digraph), let $G\le \Aut(\Gamma)$ and let $v\in \V\Gamma$. We denote by $G_v$ the {\em stabiliser} of the vertex $v$, by $\Gamma(v) = \{ u\in \V\Gamma : (v,u)\in \A\Gamma\}$ the {\em neighbourhood} of the vertex $v$ and by $G_v^{\Gamma(v)}$ the permutation group induced by $G_v$ on $\Gamma(v)$.  
Suppose now that $\Gamma$ is a $G$-arc-transitive connected graph. 
As usual, when $G=\Aut(\Gamma)$, we omit the label $G$ and we simply say that $\Gamma$ is $s$-arc-transitive. Observe that a $G$-arc-transitive graph $\Gamma$ is $(G,2)$-arc-transitive if and only if  $G_v^{\Gamma(v)}$ is a $2$-transitive permutation group.

An edge- and vertex-transitive group of automorphisms $G$ of a connected graph $\Gamma$ that is not arc-transitive is called $\frac{1}{2}$-arc-transitive. Note that
in this case $G$ possesses two orbits on arcs, each orbit containing precisely one arc underlying
each edge. If $A$ is an orbit of $G$ on the arc-set of $\Gamma$, then $(\V\Gamma,A)$ is
an arc-transitive digraph, denoted $\vGa^{(G)}$, whose underlying graph is $\Gamma$. In particular,
if $\Gamma$ has valency $4$, then the in-valence and out-valence of every vertex of $\vGa^{(G)}$ 
is $2$.

Given a set $\Omega$, we denote by $\Sym(\Omega)$ and $\Alt(\Omega)$ the symmetric and the alternating group on $\Omega$. When the domain $\Omega$ is irrelevant or clear from the context, we write $\Sym(n)$ and $\Alt(n)$ for the symmetric and alternating group of degree $n$.
Given a permutation $g\in \Sym(\Omega)$, we write $\Fix_\Omega(g)$ for the set of {\em fixed points} $\{\omega\in\Omega\mid \omega^g=\omega\}$ of $g$ and we write $\fpr_\Omega(g)$ for the {\em fixed-point-ratio} of $g$, that is $$\fpr_\Omega(g):=\frac{|\Fix_\Omega(g)|}{|\Omega|}.$$

Given $n\in\mathbb{N}\setminus\{0\}$, we denote by $\D_n$ the dihedral group of order $2n$ and we view $\D_n$ as a permutation group of degree $n$; similarly, we denote by $\C_n$ the cyclic group of order $n$. Similarly, we denote by $\mathbb{Z}_n$ the integers modulo $n$.

A subgroup $G$ of $\Sym(\Omega)$ is said to be {\em semiregular} if the identity is the only element of $G$ fixing some point of $\Omega$. Let $G$ be a group and let $H$ be a subgroup of $G$, we denote by $H\backslash G$ the {\em set of right cosets} of $H$ in $G$. Recall that $G$ acts transitively on $H\backslash G$ by right multiplication.
If $G$ is a group and $a,b\in G$, we let $[a,b] = a^{-1}b^{-1}ab$ be the commutator of $a$ and $b$, and 
${\cent G a} = \{c\in G : ca=ac\}$ be the centraliser of $g$ in $G$.

\subsection{The twelve sporadic graphs from Theorems~\ref{thrm:1} and \ref{thrm:2}}
\label{sub1}

We start by describing the six sporadic examples from Theorem~\ref{thrm:1}.
\begin{enumerate}
\item[$\Psi_1$] The first graph is the complete  graph $K_5$. The automorphism group of $K_5$ is $\Sym(5)$. A permutation of $\Sym(5)$ fixing two or three points gives rise to a non-identity automorphism fixing more than a $1/3$ of the vertices.
\item[$\Psi_2$] The second graph is the complete bipartite graph minus a complete matching $K_{5,5}- 5K_2$. The automorphism group of this graph is isomorphic to $\Sym(5)\times \C_2$. A permutation of $\Sym(5)$ fixing two or three points gives rise to a non-identity automorphism fixing four or six vertices and hence fixing  more than a $1/3$ of the vertices. Moreover,
 $\Aut(\Psi_2)$ contains a vertex-transitive copy of $\Sym(5)$ which fixes four vertices of $\Psi_2$.
\item[$\Psi_3$] The third graph arises from the Fano plane. This graph is bipartite with bipartition given by the seven points and the seven lines of the Fano plane, where the incidence in the graph is given by the anti-flags in the plane, that is, the point $p$ is adjacent to the line $\ell$ if and only if $p\notin \ell$. In other words, $\Psi_3$ is the bipartite complement of the Heawood graph. The automorphism group of this graph is isomorphic to $\Aut(\PSL_3(2))\cong \PGL_2(7)$. An involution of $\PSL_3(2)$ gives rise to a non-identity automorphism fixing  six vertices and hence fixing  more than a $1/3$ of the vertices of the graph.
 \item[$\Psi_4$] The fourth graph is similar to $\Psi_3$ and arises from the projective plane over the finite field with three elements. This graph is bipartite with bipartition given by the thirteen points and the thirteen lines of the projective plane, where the incidence in the graph is given by the flags in the plane, that is, the point $p$ is adjacent to the line $\ell$ if and only if $p\in \ell$. The automorphism group of $\Psi_4$ is isomorphic to $\Aut(\mathrm{PGL}_3(3))$. An involution of $\mathrm{PGL}_3(3)$ gives rise to a non-identity automorphism fixing  ten vertices and hence fixing  more than  $1/3$ of the vertices.
\item[$\Psi_5$] The fifth graph is a Kneser graph. This graph has $35$ vertices and these are labeled by the $35$ subsets of $\{1,\ldots,7\}$ having cardinality $3$. Two $3$-subsets $a$ and $b$ are declared to be adjacent  if and only if $a\cap b=\emptyset$. The automorphism group of this graph is isomorphic to $\Sym(7)$. A transposition of $\Sym(7)$ gives rise to a non-identity automorphism fixing  fifteen vertices and hence fixing  more than  $1/3$ of the vertices of the graph.
\item[$\Psi_6$] The sixth (and last) graph is the standard double cover of $\Psi_5$. This graph has $70$ vertices and these are labeled by the ordered pairs $(v,i)$, where $v$ is a vertex of $\Psi_5$ and $i\in \{0,1\}$. The vertices $(v,0)$ and $(w,1)$ are declared to be adjacent  if and only if $v$ and  $w$ are adjacent in $\Psi_5$. The automorphism group of this graph is isomorphic to $\Sym(7)\times \C_2$. A transposition of $\Sym(7)$ gives rise to a non-identity automorphism fixing  thirty vertices and hence fixing  more than  $1/3$ of the vertices of the graph.
Similarly as $\Psi_2$,  $\Aut(\Psi_6)$ also contains a vertex-transitive copy of the symmetric group $\Sym(7)$, however the maximum fixed point-ratio of a non-trivial element in
this group is $1/5$.
\end{enumerate}

 We now describe the six sporadic examples from Theorem~\ref{thrm:2}.
\begin{enumerate}
\item[$\Lambda_1$] The first graph is the complete  graph $K_4$. The automorphism group of this graph is $\Sym(4)$. A transposition of $\Sym(4)$   gives rise to a non-identity automorphism fixing more than  $1/3$ of the vertices of the graph.
\item[$\Lambda_2$] The second graph is the complete bipartite graph $K_{3,3}$. The automorphism group of this graph is isomorphic to $\Sym(3)\wr \Sym(2)$. A transposition from the base group $\Sym(3)\times \Sym(3)$ gives rise to a non-identity automorphism fixing four vertices and hence fixing  more than  $1/3$ of the vertices.
\item[$\Lambda_3$] The third graph is the $1$-skeleton of the cube. This graph is the Hamming graph over the $3$-dimensional vector space $\mathbb{F}_2^3$ over the field $\mathbb{F}_2$ with two elements. Two vertices $(x_1,x_2,x_3)$ and $(y_1,y_2,y_3)$ are declared to be adjacent if and only if the vectors $(x_1,x_2,x_3)$ and $(y_1,y_2,y_3)$ differ in one, and only one, coordinate.  The automorphism group of this graph is isomorphic to $\Sym(2)\wr\Sym(3)\cong \Sym(4)\times \Sym(2)$. A transposition from $\Sym(4)$ gives rise to a non-identity automorphism fixing  four vertices and hence fixing  more than $1/3$ of the vertices.
 \item[$\Lambda_4$] The fourth graph is the ubiquitous Petersen graph and  it is a Kneser graph where the $10$ vertices are the subsets of $\{1,\ldots,5\}$ having cardinality $2$. Two $2$-subsets $a$ and $b$ are declared to be adjacent  if and only if $a\cap b=\emptyset$.  The automorphism group of this graph is isomorphic to $\Sym(5)$. A transposition from $\Sym(5)$ gives rise to a non-identity automorphism fixing  four vertices and hence fixing  more than  $1/3$ of the vertices of the graph.

\item[$\Lambda_5$] The fifth graph arises from the Fano plane and it is the bipartite complement of  $\Psi_3$, that is, $\Lambda_5$ is the Heawood graph.  The automorphism group of this graph is isomorphic to $\Aut(\mathrm{GL}_3(2))\cong \mathrm{PGL}_2(7)$. An involution of $\mathrm{GL}_3(2)$ gives rise to a non-identity automorphism fixing  six vertices and hence fixing  more than  $1/3$ of the vertices of the graph.

\item[$\Lambda_6$] The sixth (and last) graph is the standard double cover of the Petersen graph. This graph has $20$ vertices and these are labeled by the ordered pairs $(v,i)$, where $v$ is a vertex of the Petersen graph and $i\in \{0,1\}$. The vertices $(v,0)$ and $(w,1)$ are declared to be adjacent  if and only if $v$ and  $w$ are adjacent in the Petersen graph. The automorphism group of this graph is isomorphic to $\Sym(5)\times \C_2$. A transposition of $\Sym(5)$ gives rise to a non-identity automorphism fixing  eight vertices and hence fixing  more than  $1/3$ of the vertices of the graph.
\end{enumerate}

\subsection{The Praeger-Xu graphs}
\label{sub2}

We now define the infinite family appearing in Theorem~\ref{thrm:2}~{\bf (ii)}. These are the ubiquitous $4$-valent \textit{Praeger-Xu graphs} $\C(r,s)$, studied in detail by Gardiner, Praeger and Xu in \cite{11,23}, and more recently in \cite{JPW}. We introduce them through their directed counterparts defined in \cite{PraHiAT}.

Let $r$ be an integer, $r\ge 3$.
Then $\vC(r,1)$ is the  the lexicographic product of a directed cycle of length $r$ with
 an edgeless graph on
$2$ vertices. In other words, $\V\vC(r, 1) = \mathbb{Z}_r \times \mathbb{Z}_2$ with 
the out-neighbours of a vertex $(x, i)$ being $(x+1,0)$ and $(x+1,1)$.
For $s\ge 2$, let $\V\vC(r,s)$ be the set of all $(s-1)$-arcs of $\vC(r,1)$ with
the out-neighbours of  $(v_0,v_1, \ldots, v_{s-1}) \in \V\vC(r,s)$ being $(v_1, \ldots, v_{s-1},u)$ and
$(v_1, \ldots, v_{s-1},u')$, where $u$ and $u'$ are the two out-neighbours of $u$ in $\vC(r,1)$.
The graph $\C(r,s)$ is then defined as the underlying graph of $\vC(r,s)$.

Clearly, $\C(r, s)$ is a connected $4$-valent graph with $r2^s$ vertices (see \cite[Theorem 2.8]{PraHiAT}).
Let us now discuss the automorphisms of the graphs $\C(r, s)$. Clearly, every automorphism of
$\vC(r,1)$ ($\C(r,1)$, respectively) acts naturally as an automorphism of
$\vC(r,s)$ ($\C(r,s)$, respectively) for every $s\ge 2$. For $i\in \ZZ_r$, let $\tau_i$ be the transposition on $\V\vC(r,1)$ swapping the vertices $(i,0)$ and $(i,1)$ while fixing every other vertex. This is clearly an
automorphism of $\vC(r,1)$, and thus also of $\vC(r,s)$ for $s\ge 2$. 
 Let
\begin{equation}
\label{eq:K}
\K:=\langle \tau_i \mid i\in \ZZ_r\rangle
\end{equation}
 and observe that
$\K \cong \C_2^r$. Further, let $\rho$ and $\sigma$ be the permutations on $\V\vC(r,1)$
defined by 
$$(x,i)^\rho := (x+1,i) \> \hbox{ and } \> (x,i)^\sigma := (x,-i).$$ 
Then $\rho$ is an automorphism of
$\vC(r,1)$, and $\sigma$ is an automorphism of $\C(r,1)$ (but not of $\vC(r,1)$). Observe that
the group $\langle \rho,\sigma\rangle$ normalises $\K$. Let 
\begin{equation}
\label{eq:H}
\rH := \K\langle\rho,\sigma\rangle\> \hbox{ and } \> \rH^+ := \K\langle\rho\rangle.
\end{equation}
Then clearly $C_2\wr \D_r \cong \rH \le \Aut(\C(r,s))$ and
$C_2\wr \C_r \cong \rH^+ \le \Aut(\vC(r,s))$ for every $r\ge 3$ and $s\ge 1$. Moreover,
$\rH$ ($\rH^+$, respectively) acts arc-transitively on $\C(r,s)$ ($\vC(r,s)$, respectively)
 whenever $1\le s \le r-1$. With three exceptions, the groups $\rH$ and $\rH^+$ are
in fact the full automorphism groups of $\C(r,s)$ and $\vC(r,s)$, respectively:

\begin{lemma}{\rm (\cite[Theorem~2.13]{23} and \cite[Theorem 2.8]{PraHiAT})}
 \label{bloodyhell}
 Let $r$, $s$, $\rH$ and $\rH^+$ be as above. Then $\Aut(\vC(r,s)) = \rH^+$ and if $r \ne 4$,
then $\Aut(\C(r,s)) = \rH$. Moreover, $|\Aut(\C(4, 1)) : \rH| = 9$, $|\Aut(\C(4, 2)) : \rH| = 3$ and
$|\Aut(\C(4, 3)) : \rH| = 2$.
\end{lemma}

\begin{remark}
\label{rem:PX2AT}{\rm
Lemma~\ref{bloodyhell} implies that $\C(r,s)$ is $2$-arc-transitive if and only if
$r=4$ and $s\in \{1,2\}$.}
\end{remark}

Let $v$ be a vertex of $\vC(r,s)$ which as an $(s-1)$-arc of $\vC(r,1)$
starts in $(x,0)$ or $(x,1)$ for some $x\in \ZZ_r$. Observe that then
$\Aut(\vC(r,s))_v = \langle \tau_i \mid i\in \ZZ_r\setminus \{x,x+1, \ldots,x+s-1\}\rangle$, 
showing that
\begin{equation}
\label{eq:KH+}
\K
\>\> = \>\> 
 \langle \Aut(\vC(r,s))_v \mid v \in \V\vC(r,s) \rangle
\>\> = \>\> 
\langle (\rH^+)_v \mid v \in \V\vC(r,s) \rangle 
\>\> =\>\> 
 \langle \K_v \mid v \in \V\vC(r,s) \rangle.
\end{equation}

The following result explains the restriction on $r$ and $s$ in Theorem~\ref{thrm:1}~{\bf (ii)}
and characterises the automorphisms of $\C(r,s)$ that fix more than $1/3$ of the vertices.

\begin{lemma}
\label{lem:rs}
The graph $\C(r,s)$ with $r\ge 3$ and $1\le s \le r-1$ contains
a non-identity automorphism fixing more than $1/3$ of the vertices if and only
$s < 2r/3$. In this case all such automorphisms belong to the group $\K$ defined by~\eqref{eq:K}.
\end{lemma}

\begin{proof}
Let $\Gamma = \C(r,s)$.
For $r\le 6$ the claim of the lemma can be verified by inspecting the $14$ graphs $\C(r,s)$, $3\le r\le 6$, $1\le s \le r-1$,
with a computer algebra system such as Magma~\cite{magma}. We may thus assume that $r\ge 7$.
By Lemma~\ref{bloodyhell}, we see that $\Aut(\Gamma) = \rH$.

Let $g$ be an arbitrary automorphism of $\Aut(\Gamma)$ fixing more than $1/3$ of the vertices.
Then $g = \tau\rho^i\sigma^\epsilon$ for some $\tau \in \K$, $i\in \ZZ_r$ and $\epsilon\in \{0,1\}$. 
For $x\in \ZZ_r$, let $\Delta_x$ be the set of $(s-1)$-arcs of $\vC(r,1)$ that start at a vertex $(x,0)$ or $(x,1)$,
and consider the elements of $\Delta_x$ as vertices of $\Gamma$.
Observe that each  $\Delta_x$ is an orbit of the action of $\K$ on $\V\Gamma$, that $(\Delta_x)^\rho = \Delta_{x+1}$
and that $(\Delta_x)^\sigma = \Delta_{-x-s+1}$. Consequently, $(\Delta_x)^g$ is either $\Delta_{x+i}$ if $\epsilon = 0$ 
or $\Delta_{-x-i-s+1}$ if $\epsilon = 1$. In particular, unless $i=0$ and $\epsilon = 0$, $g$ preserves at most $2$ out of $r$
orbits $\Delta_x$, $x\in \ZZ_r$. Since $r\ge 7$ and since $g$ fixes more then $1/3$ vertices of $\Gamma$, this implies
that $i=0$, $\epsilon = 0$, and thus that $g\in \K$.

 Finally, note that $\tau_i$ moves precisely those $(s-1)$-arcs  of $\vC(r,1)$ that pass through one of the vertices $(i,0)$ or $(i,1)$. 
 Therefore, $\tau_i$, as an automorphism of $\C(r,s)$, fixes all but $s2^{s}$ vertices.
Similarly, an element
$
\prod_{i\in J} \tau_i \in \K
$
moves precisely those $(s-1)$-arcs  of $\vC(r,1)$ that pass through at least one of the vertices $\{(i,\epsilon) : \epsilon \in \{0,1\}, i\in J\}$,
implying that such an element fixes at most as many elements as a single $\tau_i$. Hence
$$
\frac{1}{3} < \fpr_{\V\C(r,s)}(g) \le \fpr_{\V\C(r,s)}(\tau_i) = \frac{(r-s)2^s}{r2^s} = \frac{r-s}{r}
$$
and the result follows.
\end{proof}

The Praeger-Xu graphs can be characterised by an existence of an abelian normal subgroup
not acting semiregularly on the vertices.
The following result appeared as Theorem~1 in \cite{23} for the case of $G$ acting arc-transitively,
and as Theorem~2.9 in \cite{PraHiAT} for the $\frac{1}{2}$-arc-transitive case.
\begin{lemma}{\rm (see \cite[Theorem~1]{23} and \cite[Theorem 2.9]{PraHiAT})}
\label{cor:3.4}
Let $\Gamma$ be a connected $4$-valent graph and let $G$ be an edge- and vertex-transitive group of automorphisms of $\Gamma$.
If $G$ has an abelian normal subgroup which is not semiregular on the vertices of $\Gamma$, then $\Gamma\cong \C(r,s)$ with $r\ge 3$ and $1\le s\le r-1$. 
\end{lemma}

The following lemma is a generalisation of \cite[Lemma 3.1]{11} from the case of $G$
 being arc-transitive group to the case of an arbitrary edge- and vertex-transitive group $G$. 
The proof closely follows that of \cite[Lemma~3.1]{11}.

\begin{lemma}
\label{lem:GNc}
Let $\Gamma$ be a connected $4$-valent graph, let $G$ be an edge- and vertex-transitive group of automorphisms of $\Gamma$ and let $N$ be a minimal normal subgroup of $G$. 
If $N$ is a $2$-group and $\Gamma/N$ is a cycle, then $\Gamma\cong\C(r,s)$ for some $r$ and $s$.
\end{lemma}

\begin{proof}
If $N$ does not act semiregularly on $\V\Gamma$, then the result follows by Lemma~\ref{cor:3.4}. We may thus assume that $N$ is semiregular.
If $G$ is arc-transitive, then result follows directly from \cite[Lemma 3.1]{11}.
We may thus assume that $G$ is $\frac{1}{2}$-arc-transitive.

Let $K$ be the kernel of the action of $G$ on the vertex-set $\{u^N : u \in \V\Gamma\}$ of
$\Gamma/N$ and let $C=C_K(N)$ be the centraliser of $N$ in $K$. Note that both $K$ and $C$ are normal in $G$.
Since $N$ is abelian, we have $N\le C$.
The stabiliser $C_u$ clearly fixes every vertex of $u^N$, showing that the group
 $C^{u^N}$ induced by the action of $C$ on $u^N$ is regular, implying that
 $C^{u^N} = N^{u^N}$ and so $C^{u^N}$ is abelian. But every permutation group
 is isomorphic to a subgroup of the direct product of the permutation groups it induces on its orbits.
 In particular, $C$ is isomorphic to a subgroup  $C^{u_1^N} \times \cdots \times C^{u_m^N}$
 where $u_1^C,\ldots,u_m^C$ are the $C$-orbits on $\V\Gamma$, implying that $C$ is abelian.
 If $C_u \not =1$ for some $u\in \V\Gamma$, then by Lemma~\ref{cor:3.4}, $\Gamma\cong\C(r,s)$.
 We may thus assume that $C_u  =1$, and therefore that $C=N$.
 
Let $v\in \V\Gamma$ and let $u,w$ be the out-neighbours of $v$ in the digraph $\vGa^{(G)}$.
Since $G$ acts arc-transitively on $\vGa^{(G)}$, there is an element $h\in G_v$ swapping the vertices $u$ and $w$.
Moreover, such an element clearly preserves each $N$ orbit, implying that $g\in K$. In particular,
$K_v\not = 1$, and so $N$ is a proper subgroup of $K$. On the other hand, since $v^N = v^K$, we have $K=NK_v$, implying that $K$ is a $2$-group. The action of $K$ on the set $N\setminus\{1\}$ of odd cardinality by conjugation thus has at least one fixed point, say $x\in N$. But then $x$ is centralised by
every element in $K$ and thus $x\in \Z K$. In particular $\Z K\cap N$ is a non-trivial normal subgroup
of $G$ contained in $G$. By minimality of $N$, it follows that $N \le \Z K$. But then $C=K$, contradicting the fact that $C_v = 1$. This contradiction shows that $\Gamma\cong\C(r,s)$ as claimed.
\end{proof}

\subsection{Split Praeger-Xu graphs}
\label{sec:subS}

We now define the family of the {\em Split Praeger-Xu graphs} $\Spl(\C(r,s))$, featuring in Theorem~\ref{thrm:2}~{\bf (ii)}. This family is obtained from the Praeger-Xu graphs via the {\em splitting operation} $\mathrm{S}(-)$, which was introduced in~\cite[Construction 9]{census1}. 
Rather then defining $\mathrm{S}(-)$ in its full generality here, we only describe it in the special case of the Praeger-Xu graphs $\C(r,s)$. We refer the reader to \cite[Section 4]{census1} for
more information on this operator.

Split each vertex $v$ of $\vGa:=\vC(r,w)$ into two copies, denoted $v_+$ and $v_-$, and let $v_-$ be adjacent to $v_+$ and to $u_+$ whenever $u$ is an in-neighbour of $v$ in $\vGa$.
Similarly, let $v_+$ be adjacent to $v_-$ and to $w_-$ whenever $w$ is an out-neighbour
 of $v$ in $\Gamma$. The resulting $3$-valent graph is then called the {\em Split Praeger-Xu graph}
and denoted $\Spl(\C(r,s))$. Observe that the automorphism group $\Aut(\vC(r,s))$ ($=\rH$)
acts faithfully as a vertex- but not arc-transitive group of automorphisms of $\Spl(\C(r,s))$
and that for every $g\in \rH$ we have $\fpr_{\V\Spl(\C(r,s))}(g) = \fpr_{\V\C(r,s)}(g)$.

\subsection{Normal quotients}

The proofs of the main theorems are inductive with the induction step using the notion of a normal quotient of a graph introduced in \cite[Section 4]{11_1}.

\begin{definition} Let $\Gamma$ be a connected 
graph (or digraph) and let $N \le \Aut(\Gamma)$. The {\em normal quotient} $\Gamma/N$ is the graph (or digraph) whose vertices are the $N$-orbits on $\V\Gamma$ with two distinct such $N$-orbits $v^N$ and $u^N$ forming an arc $(v^N,u^N)$ of $\Gamma/N$ whenever there is a pair of vertices $v'\in v^N$ and $u'\in u^N$ such that $(v',u')$ is an arc of $\Gamma$. 
\end{definition}

If the group $N$ is normalised by some group $G\le \Aut(\Gamma)$, then $G/N$ acts (possibly unfaithfully)  on $\Gamma/N$ as a group of automorphisms. If $G$ is vertex-, edge- or arc-transitive on $\Gamma$, then so is $G/N$ on $\Gamma/N$.
Suppose now that $\Gamma$ is a $4$-valent graph and that $G$ is arc-transitive. Then
the valency of $\Gamma/N$ is either $0$ (when $N$ is transitive on $\V\Gamma$), $1$ (when $N$ has $2$ orbits on $\V\Gamma$), $2$ (when $\Gamma/N$ is a cycle) or $4$. In the latter case, $G/N$ acts faithfully on $\V\Gamma$ and hence $\Gamma/N$ is a connected $4$-valent $G/N$-arc-transitive graph   with vertex-stabiliser $(G/N)_{v^N}=G_vN/N$ in $G/N$ isomorphic to $G_v$. Moreover, if $\Gamma/N$ is $4$-valent, $N_v=1$ for every vertex $v\in\V\Gamma$.

The following lemma follows almost immediately from the theory of lifting automorphisms
along covering projections, as developed in \cite{MNS00}. However, in order to avoid leading the reader astray with introducing these methods, we decided to provide a straightforward, though longer proof.

\begin{lemma}\label{4.6}
Let $\Gamma$ be a connected $4$-valent 
$G$-arc-transitive 
graph and let $N$ be a semiregular normal subgroup of $G$ such that the normal quotient $\Gamma/N$ is a cycle of length $r \ge 3$. Let $K$ be the kernel of the action
of $G$ on the $N$-orbits on $\V\Gamma$. Then $K_v$ is an elementary abelian $2$-group.
\end{lemma}

\begin{proof}
Let $\Delta_0, \Delta_1, \ldots, \Delta_{r-1}$ be the orbits of $N$ in its action on $\V\Gamma$.
Since $\Gamma/N$ is a cycle, we may assume that $\Delta_i$ is adjacent to $\Delta_{i-1}$ and
$\Delta_{i+1}$ with indices computed modulo $r$. Since $N$ is normal in $G$, the orbits of $N$ on the edge-set $\E\Gamma$ form a $G$-invariant partition of $\E\Gamma$. Since $N$ acts semiregularly on $\V\Gamma$, no two edges incident to a fixed vertex of $\Gamma$ belong
to the same $N$-edge-orbit. Moreover, since $G$ is arc-transitive, every vertex $v\in \Delta_i$
is adjacent to two vertices in $\Delta_{i-1}$ and two vertices in $\Delta_{i+1}$, implying that
the edges between $\Delta_i$ and $\Delta_{i+1}$ are partitioned into precisely
two $N$-edge-orbits; let's call these two orbits $\Theta_i$ and $\Theta_i'$. 

Clearly, an element of $K$ can map an edge in $\Theta_i$ only to an edge in $\Theta_i$ or
to an edge in $\Theta_i'$. On the other hand, for every vertex $v\in \Theta_i$ there is
an element $g\in G_v$ which maps an edge of $\Theta_i$ incident to $v$ to the edge of $\Theta_i'$ incident to $v$; and this element $g$ is clearly an element of $K$. This shows that
 the orbits of $K$ on $\E\Gamma$ are precisely the sets $\Theta_i \cup \Theta_i'$, $i\in \ZZ_r$.
 In other words, each orbit of the induced action of $K$ on the set $\E\Gamma/N = \{e^N : e\in \E\Gamma\}$ has length $2$. Consequently, if $X$ denotes the kernel of the action of $K$ on $\E\Gamma$, then $K/X$ embeds into $\Sym(2)^r$ and is therefore an elementary abelian $2$-group.
 
 Let us now show that $X=N$. Clearly, $N\le X$. 
 Let $v\in \Delta_0$. Since $N$ is transitive on $\Delta_0$,
 it follows that $X=NX_v$. Suppose that $X_v$ is non-trivial and let $g$ be a non-trivial 
 element of $X_v$. Further, let $w$ be a vertex which is closest to $v$ among all the vertices not fixed by $g$, and let $v=v_0\sim v_1 \sim \ldots \sim v_m = w$ be a shortest path from $v$ to $w$.
 Then $v_{m-1}$ is fixed by $g$. Since $g$ fixes each $N$-edge-orbit set-wise and since every
 vertex of $\Gamma$ is incident to at most one edge in each $N$-edge-orbit, it follows that
 $g$ fixes all the neighbours of $v_{m-1}$, thus also $v_m$. This contradicts our assumptions
 and proves that $X_v$ is a trivial group, and hence that $X=N$.
 
Thus $K/N$ is an elementary abelian $2$-group. Now, since $N$ is semiregular,
 we see that $K_v\cong K_v/(N\cap K_v) \cong K_vN/N = K/N$ (the latter equality following from the
 fact that $N$ is transitive on $v^K$). Hence, $K_v$ is an elementary abelian $2$-group, as claimed.
\end{proof}

\subsection{Miscellanea}
We now present several auxiliary results that will come useful when proving Theorems~\ref{thrm:1} and~\ref{thrm:2}.

\begin{lemma}
\label{lem:2orb}
Let $\Gamma$ be a connected $k$-valent graph admitting an abelian
group of automorphisms $N$ having at most two orbits on $\V\Gamma$ and let $v\in\V\Gamma$. Then either $N_v\not =1$ and there exist two distinct vertices $u$ and $u'$ with $\Gamma(u) = \Gamma(u')$,
 or $N_v = 1$, $N$ has a generating set consisting of at most $2k-2$ elements and $|\V\Gamma| = 2|N|$.
\end{lemma}

\begin{proof}
Observe that since $N$ is abelian,  $N_v = N_u$ whenever $u \in v^N$. Hence,
if $N$ is transitive on $\V\Gamma$, then $\Gamma$ is a Cayley graph on the group $N$, and $N$ is generated by the $k$-element set $\{x\in N : v^x \sim_\Gamma v\}$.

Suppose now that $N$ has two orbits on $\V\Gamma$ and let $u$ be a neighbour of $v$. 
If $N_v =1$,
then $\Gamma$ is isomorphic to a bi-Cayley graph ${\rm BiCay}(N;L,R,S\})$ (see, for example, \cite{bicayley} for the definition of bi-Cayley graphs and basic properties) where $S,L,R \subseteq N$, $1\in S$, $|S| + |L| = |S| + |R| = k$ and $\langle S\cup L\cup R \rangle = N$. In particular, $N$ has a generating set of size at most $2k-2$.
If $N_v\not =1$, then $N_v$ and $N_u$ are kernels of the action of $N$ on $v^N$ and $u^N$, respectively, and thus $N_v\cap N_u = \emptyset$.
Hence, if $x$ is a non-trivial element of $N_v$, then $u^x \not = u$ and $\Gamma(u)=\Gamma(u^x)$.
\end{proof}
\color{black}

\begin{lemma}\label{lemma:technical}
Let $\Gamma$ be one of the exceptional graphs $\Psi_1, \ldots, \Psi_6$ or $\C(r,s)$ with $r\ge s$
and $1\le s\le r-1$ and let $G$ be an edge- and vertex-transitive group of automorphisms of $\Gamma$
containing a non-trivial element $g$ with $\fpr_{\V\Gamma}(g) > 1/3$. Then exactly one of the following happens:
\begin{enumerate}
\item\label{te:part1} $G$ is $2$-arc-transitive, or
\item\label{te:part2} $\Gamma\cong \C(r,s)$ with $1\le s \le 2r/3$ and
$G$ is $\Aut(\Gamma)$-conjugate to a subgroup of $\rH$ defined in~\eqref{eq:H}.
\end{enumerate}
\end{lemma}

\begin{proof}
Suppose first that $\Gamma\cong \C(r,s)$. Then Lemma~\ref{lem:rs} implies that
$s<2r/3$. If $r\not = 4$, then, by Lemma~\ref{bloodyhell}, we see that $\Aut(\Gamma)=\rH$ and
the result follows (note that $\rH$ is not $2$-arc-transitive). Suppose now that $r=4$, and thus
 $s\in \{1,2\}$. Since $|\rH| = r2^{r+1}$, we see that $\rH$ is a $2$-group, and by Lemma~\ref{bloodyhell}, $|\Aut(\Gamma):\rH|=3$ or $9$. Hence $\rH$ is a Sylow $2$-subgroup of $\Aut(\Gamma)$. If $G$ is not $2$-arc-transitive, then $G_v$ is a $2$-group and since
 $|\V\Gamma| = 8$ or $16$, we see that $G$ is a $2$-group, implying that $G$ is conjugate
 to a subgroup of the Sylow $2$-subgroup $\rH$.
For $\Gamma\cong \Psi_i$ with $i\in \{1,\ldots,6\}$ we verified the
claim of the lemma by using the algebra computational system {\sc Magma} \cite{magma}.
\end{proof}

\begin{lemma}
\label{lem:quofix}
Let $G$ be a group acting transitively on the set $\Omega$ and let $\Sigma$ be a $G$-invariant partition
of $\Omega$. For $g\in G$, let $g^\Sigma$ be the permutation of $\Sigma$ induced by $g$. Then
$$%
\fpr_{\Omega}(g) \le \fpr_{\Sigma}(g^\Sigma).
$$%
In particular, if $N$ is a normal subgroup of $G$, then
$
\fpr_{\Omega}(g) \le \fpr_{\Omega/N}(Ng).
$
\end{lemma}

\begin{proof}
Observe that if $\omega\in\Fix_\Omega(g)$ and $[\omega]$ is the element of $\Sigma$ containing $\omega$, then $[\omega] \in \Fix_{\Omega/N}(Ng)$. 
Hence
$$%
|\Fix_{\Omega}(g)| = \sum_{B\in\Fix_{\Sigma}(g^\Sigma)} |B \cap \Fix_{\Omega}(g)| \le b |\Fix_{\Sigma}(g^\Sigma)|,
$$%
where $b$ is the cardinality of an arbitrary element of $\Sigma$. Note that
$b|\Sigma|= |\Omega|$.
The claim of the lemma  then follows by dividing the above inequality by $|\Omega|$
and observing that $\Omega/N$ is a $G$-invariant partition of $\Omega$.
\end{proof}

 The following lemma was proved in \cite[Lemma 2.2]{PotSpiO2} and is just a slight generalisation of \cite[Lemma~2.5]{LS}.

\begin{lemma}
\label{eq:22}
\cite[Lemma 2.2]{PotSpiO2}
Let $X$ be a group acting on a set $\Omega$, let $Y$ be a normal subgroup of $X$, let
$\omega\in \Omega$ and let $x\in X_\omega$. Then
\begin{equation}\label{eq:2}
\fpr_{\omega^Y}(x)=\frac{|x^Y\cap X_\omega|}{|x^Y|} = \frac{|x^Y\cap X_\omega|}{|Y:\cent Y x|} ,
\end{equation}
where $x^Y:=\{x^y\mid y\in Y\}$ is the $Y$-conjugacy class of the element $x$ and 
$\cent Y x$ the centraliser of $x$ in $Y$.
\end{lemma}

If $B\le G\le \Sym(\Omega)$ and $\omega\in\Omega$, we let $[a,B] = \{[a,b] : b\in B\}$ and
$[a,B]_\omega = [a,B] \cap G_\omega$. Note that $[a,B]$ is not necessarily a subgroup of $G$.
Observe that if $B$ is semiregular and normalised by $a$, then $[a,B]_\omega = 1$.

\begin{lemma}
\label{lemma:1}
Let $G\le \Sym(\Omega)$, let $\omega\in \Omega$, let $g\in G_\omega$,
and let $X$ be a normal subgroup of $G$ such that $[g,X]_\omega = 1$. Then
$$
\Fix_{\omega^X}(g)=\omega^{\cent X g}\quad \hbox{and} \quad
\fpr_\Omega(g) \le \fpr_{\omega^X}(g) =
\frac{1}{|X : \cent X g|}.
$$
\end{lemma}

\begin{proof}
Observe first that $1=[g,X]_\omega = \{[g,x] : x\in X, [g,x]\in G_\omega\} \supseteq \{[g,x] : x\in X_\omega\}$,
implying that every $x\in X_\omega$ centralises $g$ and so $X_\omega = {\cent X g}_\omega$.
Let $\delta\in \Fix_{\omega^X}(g)$. Then there exists $x\in X$ with $\delta^x=\omega$ and 
$\omega^{g^{-1}x^{-1}g}=(\omega^{x^{-1}})^g=\delta^g=\delta=\omega^{x^{-1}},$
implying that $\omega^{[g,x]} = \omega$.
Since $[g,X]_\omega = 1$, this shows that $x \in \cent X g$ and hence
that $\Fix_{\omega^X}(g)=\omega^{\cent X g}$, as claimed. Therefore
$|\Fix_{\omega^X}(g)| = |{\cent X g : \cent X g}_\omega|$ and so
$$
\fpr_{\omega^X}(g) = \frac{|\Fix_{\omega^X}(g)|}{|\omega^X|} 
= \frac{|\cent X g|\, |X_\omega|}{|X| |{\cent X g}_\omega|} 
= \frac{1}{|X : \cent X g|} ,
$$
as claimed.
In particular, $\fpr_{\omega^g}(g) =  \fpr_{\delta^g}(g)$ for every
$\delta \in \Fix_\Omega(g)$.
Now choose a set $\{\delta_1, \ldots, \delta_m\}$ of representatives of
those orbits $\delta^X$ for which $\Fix_{\delta^X}(g) \not=\emptyset$.
Without loss of generality, we may assume that all $\delta_i$ are fixed by $g$.
The set $\Fix_\Omega(g)$ is then the disjoint union of the sets $\Fix_{\delta_i^X}(g)$
for $i\in\{1,\ldots,m\}$.
Since $|\Omega| = |\Omega/X|\, |\omega^X|\ge m |\omega^X|$,
 we see that 
$$ \fpr_{\Omega}(g)=
   \frac{|\Fix_{\Omega}(g)|}{|\Omega|} 
  \le    \frac{|\Fix_{\Omega}(g)|}{m |\omega^X|} 
   = \frac{1}{m} \sum_{i=1}^m \frac{|\Fix_{\delta_i^X}(g)|}{|\omega^X|}
      = \frac{1}{m} \sum_{i=1}^m \fpr_{\delta_i^X}(g)
    = \fpr_{\omega^X}(g),
$$ 
completing the proof.
\end{proof}


\section{Proof of Theorem~\ref{thrm:1} for $\Gamma$ not $2$-arc-transitive}

This section is devoted to the proof of Theorem~\ref{thrm:1} in the case
where $\Gamma$ is not a $2$-arc-transitive graph.
Throughout this section, we work under the following assumption:

\begin{hypothesis}
\label{hyp:init}
{\rm
Let $\Gamma$ be a connected $4$-valent graph, let $G$ be a subgroup of $\Aut(\Gamma)$ acting
transitively on $\V\Gamma$ and on $\E\Gamma$ and let $g$ be a non-trivial element of $G$ such that $\fpr_{\V\Gamma}(g)>1/3$.
}
\end{hypothesis}

\begin{lemma}
\label{lem:GHAT}
Let $\Gamma$, $G$ and $g$ be as in Hypothesis~\ref{hyp:init}.
If $G$ is not $2$-arc-transitive,
then either $\Gamma \cong \C(r,s)$ for some integers $r$ and $s$,
or $G$ contains a minimal normal subgroup $N$ of order a power of $2$
such that $\Gamma/N$ is a $4$-valent graph.
\end{lemma}

\begin{proof}
First, observe that since $G$ is not $2$-arc-transitive,
the vertex-stabiliser $G_v$ is a $2$-group (for example, see \cite{PSV} for the arc-transitive case and
\cite{PVdigraphs} for the $\frac{1}{2}$-arc-transitive case).
If $G$ possesses no non-trivial normal $2$-subgroups, then
Theorem~\ref{proposition:7} yields a contradiction.
We may thus assume that $G$ has a minimal normal subgroup
$N$ which is a $2$-group (and hence an elementary abelian $2$-group).
If $N_v\not = 1$, then by Lemma~\ref{cor:3.4}, we have $\Gamma \cong \C(r,s)$.
We may thus assume that $N_v =1$.
If $N$ has at most two orbits on $\V\Gamma$, then by Lemma~\ref{lem:2orb}, we see that
$|\V\Gamma| \le 128$ and the validity of the claim can be checked computationally
by inspecting
the candidate graphs from the list of all $4$-valent arc-transitive graphs of small order
(see \cite{census1} or \cite{recipe}) or small $4$-valent $\frac{1}{2}$-transitive graphs \cite{census2}.
We may thus assume that
 $N$ has at least three orbits on $\V\Gamma$, and therefore
$\Gamma/N$ is either a cycle $\C_r$ for some $r\ge 3$ or a $4$-valent graph. In the former case,
Lemma~\ref{lem:GNc} implies that $\Gamma\cong \C(r,s)$, and the result follows.
\end{proof}

We now prove a result that will enable us to reduce the proof of Theorem~\ref{thrm:1} to the case where $G$ is $2$-arc-transitive.

\begin{lemma}
\label{lem:vCcov}
Let $\Gamma$, $G$ and $g$ be as in Hypothesis~\ref{hyp:init}.
Suppose that $G$ contains a minimal normal subgroup $N$ of order a power of $2$
such that $\Gamma/N$ is isomorphic to $\C(r,s)$ for some $r$ and $s$.
If $G/N \le \rH^+$ where $\rH^+$ is as in the formula {\rm (\ref{eq:H})} of Section~\ref{sub2},
then $\Gamma\cong \C(r',s')$ with $1\le s' \le 2r'/3$.
\end{lemma}

\begin{proof}
By Lemma~\ref{lem:rs}, we have $1\le s \le 2r/3$. Moreover, by way of contradiction, we
may assume that
$\Gamma\not\cong \C(r',s')$ for any $r'$ and $s'$.
Since $G/N \le \rH^+$, the group $G/N$ is not arc-transitive 
 and thus neither is $G$.
Let $\vGa:=\vGa^{(G)}$ be a digraph induced by the $\frac{1}{2}$-arc-transitive action of $G$.
Then $\vGa$ is a $G$-arc-transitive digraph of in- and out-valency $2$, and thus
in view of Lemmas~\ref{bloodyhell} and \ref{cor:3.4}, 
we may assume that:
\begin{equation}
\label{eq:no}
\hbox{Every abelian normal subgroup of } G\hbox{ acts semiregularly on } \V\vGa.
\end{equation}
Note that $\vGa/N \cong \vC(r,s)$.
We will now follow the ideas developed in the proof of~\cite[Theorem~$3.10$]{PSV}
as well as in \cite{PVdigraphs}, which in turn
draw heavily from the classical work of Tutte~\cite{tutte},  Djokovovi\'c~\cite{djokovic} and Sims~\cite{sims}. 

Identify $\vGa/N$ with $\vC(r,s)$ and let the automorphisms $\tau_i, \rho \in \Aut(\vGa)$ and the group $\K = \langle \tau_i \mid i\in \ZZ_r\rangle \cong \C_2^r$ be as in Section~\ref{sub2}.
Recall that $\rH^+ = \K\langle \rho\rangle$ and that $\K = \langle (\rH^+)_{x} : x \in \V\vC(r,s) \rangle = \langle \K_{x} : x \in \V\vC(r,s) \rangle$ (see formula~(\ref{eq:KH+})).
 Since $G/N \le H^+$ this implies that $(G/N)_{v^N} \le \K_{v^N}$ and hence
 $(G/N)_{v^N} = (G/N) \cap \K_v$.
Now let
\begin{equation}
E = \langle G_u : u \in \V\vGa\rangle
\end{equation}
and observe that $E=(G_v)^G$, the normal closure of $G_v$ in $G$.
We now see that
\begin{equation}
EN/N = (G_vN)^G/N = (G_vN/N)^{G/N} = ((G/N)_{v^N})^{G/N} = ((G/N) \cap \K_v)^{G/N}
= (G/N) \cap \K.
\end{equation}
By minimality of $N$, it follows that either $N\le G$ or $N\cap E = 1$.
If $N\cap E = 1$, we see that $E \cong EN/N \le \K$;
in particular, $E$ is an abelian normal subgroup of $G$ not acting semiregularly on $\V\vGa$,
 contradicting (\ref{eq:no}).

We may thus assume that $N\le E$. Then $E/N = EN/N = (G/N) \cap \K$.
Hence $E/N$ is an elementary abelian $2$-groups, implying that $E$ is
a $2$-group. Moreover, since $E_u \cong E_u /(N\cap E_u) \cong (E_uN/N) = (E/N)_{u^N}$,
wee see that
\begin{equation}
\label{eq:Euab}
E_u \> \hbox{ is an elementary abelian } 2\hbox{-group }\hbox{for every } u\in \V\Gamma.
\end{equation}
Now consider the Frattini subgroup $\Phi(E)$ and the derived subgroup
$[E,E]$ of $E$. Being characteristic in $E$, they are both normal in $G$. Recall that $\Phi(E)$  (respectively, $[E,E]$) is the smallest normal subgroup of $E$ with respect to which the quotient group is elementary abelian (respectively, abelian).
 In particular, since $E/N$ is elementary abelian, we see that $[E,E]\le\Phi(E)\le N$. By the minimality of $N$, it follows
that either $[E,E]=1$ or $[E,E] = \Phi(E) = N$. If $[E,E] = 1$, then $E$ is abelian, which contradicts (\ref{eq:no}).
We may thus assume that 
\begin{equation}
\label{eq:EE}
[E,E] = \Phi(E) = N.
\end{equation}

Moreover, since $E$ and $N$ are $2$-groups, the action of $E$ on $N\setminus\{1\}$ by conjugation must have at least one fixed point, implying that $N$ intersects the centre  $\Z E$ non-trivially. The minimality of $N$ then implies that $\Z E\ge N$. 
If $(\Z E)_v\ne 1$, then $\Z E$ is a normal elementary abelian $2$-subgroup of $G$ not acting semiregularly on $\V\Gamma$, which contradicts (\ref{eq:no}), showing that $\Z E$ acts semiregularly on $\V\Gamma$.

We will now set up a standard notation typically used when studying 
the structure of a vertex-stabiliser $G_v$ in a $G$-arc-transitive digraph of out-valence $2$
(see, for example, \cite[Section 2.3]{PVdigraphs}).
Let $t$ be the largest integer such that $G$ acts transitively on the $t$-arcs of $\vGa$. 
Note that $G$ then acts regularly on the set of all $t$-arcs of $\vGa$ and that
 $t$ is the largest integer such that $G_v$ (which clearly equals $E_v$) acts transitively on the $t$-arcs starting at $v$.
Let $a$ be any element of $G$ such that $(v^a,v)$ is an arc of $\vGa$ and let 
\begin{equation}
\label{eq:vi}
v_i := v^{a^{-i}} \hbox{  for } i\in \ZZ.
\end{equation}
Note that, for every $i\ge 0$, the $(i+1)$-tuple $(v_0,v_1, \ldots, v_i)$ is an $i$-arc of $\vGa$.
Observe also that $Ea \in G/E$ acts as a one-step rotation of $\vGa/E \cong \vec{\C}_r$, implying that $G=E\langle a\rangle$.

 Now consider the stabiliser
$G_{(v_0, \ldots, v_i)}$ of this $i$-arc in $G$; since $v_0 = v$ and $G_v = E_v$, we see that
$G_{(v_0, \ldots, v_i)} = E_{(v_0, \ldots, v_i)}$.
 By the definition of $t$, it follows that $G_{(v_0, \ldots, v_{t})} = 1$ and that $G_{(v_0, \ldots, v_{t-1})} = \langle x_0\rangle$,
where $x_0$ is the unique element of $G$ which fixes the $t-1$ arc $(v_0,v_1, \ldots, v_{t-1})$ but moves the vertex $v_t$.
For $i\ge 1$, let
\begin{equation}
\label{eq:xi}
x_i := x^{a^{i}}\>\> \hbox{  and } \>\> E_i:= \langle x_0, \ldots, x_{i-1} \rangle,\>\> E_0:=1.
 \end{equation}
It is not difficult to deduce (see \cite[Section 2.3]{PVdigraphs} for the proof) that
\begin{equation}
\label{eq:Ei}
\hbox{for every } i\in\{0, \ldots, t\}:\>\>
  E_i = G_{(v_0, \ldots, v_{t-i})}\> \hbox{ and }\> |E_i| = 2^i.
 \end{equation}
Moreover, by \cite[Section 2.3]{PVdigraphs},
 there exists a positive integer $e$ with the following properties:
\begin{itemize}
\item $e$ is the smallest integer such that $E_{t+e} = E_{t+e+1}$;
\item $e$ is the smallest integer such that $E_{t+e} = E$.
\end{itemize}
Recall that $E/N = (G/N) \cap \K$ is an elementary abelian $2$-group.
Let us now show that
\begin{equation}
\label{eq:E/N}
 |E/N| = 2^{t+e}\>\> \hbox{ and }\>\> E/N = \langle Nx_0, Nx_1, \ldots, Nx_{t+e-1} \rangle.
\end{equation}
Indeed: 
Suppose that for some $e' \le e$ we have
$E_{t+e'}N = E_{t+e'+1}N$. Since $(E_i)^a = \langle x_1, \ldots, x_i\rangle$,
it follows that
$E_{t+e'+2}N=\langle E_{t+e'+1}N,(E_{t+e'+1}N)^a\rangle =$
$\langle E_{t+e'}N,(E_{t+e'}N)^a\rangle = E_{t+e'+1}N = E_{t+e'}N$,
and thus by induction we see that 
 $E=E_{t+e}N = E_{t+e'}N$.
Since $N = \Phi(E)$, the set of non-generators of
$E$, it follows that $E_{t+e'} =E$, and thus $e=e'$. 
In particular,
 $e$ is the smallest integer such that $E_{t+e}N = E_{t+e+1}N$.
Hence
\begin{equation}
 N=E_0N<E_1N<E_2N<\cdots<E_{t+e-1}N<E_{t+e}N=E.
\end{equation}
In particular, $|E| \ge 2^{t+e}|N|$ and thus $|E/N| \ge 2^{t+e}$.
On the other hand, $E/N = \langle Nx_0, Nx_1, \ldots, Nx_{t+e-1}\rangle$,
and since $E/N$ is elementary abelian,
we see that $|E|\le2^{t+e}$, proving the claim (\ref{eq:E/N}).

Recall that $\Z E$ acts semiregularly on $\V\Gamma$. Since $E_v=E_t=\langle x_0,\ldots,x_{t-1}\rangle$ is abelian (see (\ref{eq:Euab})), it follows that $E_t^{a^{t-1}}=\langle x_{t-1},\ldots,x_{2t-2}\rangle$ is also abelian. Therefore $x_{t-1}$ is central in $\langle E_t,(E_t)^{a^{t-1}}\rangle=\langle x_0,\ldots,x_{2t-2}\rangle=E_{2t-1}$. Since $x_{t-1}\in E_v$ and $\Z E\cap E_v=1$, we get $E_{2t-1}<E=E_{t+e}$ and hence $2t-1<t+e$ from which it follows that 
\begin{equation}
\label{eq:mygod1}
e\geq t.
\end{equation}

We now prove a technical result (its relevance will become apparent later in the proof). Let $h$ be an arbitrary element of
$E_{v_0}\setminus E_{v_{-1}}$ and let $\beta = \min \{b \in \NN : h\in G_{(v_0, \ldots, v_{t-b})}\}$. Then:
\begin{equation}
\label{eq:tech}
1\le \beta \le t
\hbox{ and for every } j\in \{1, \ldots, e\} \hbox{ and every } x\in E \hbox{ we have } h^{a^{t-\beta+j}x} \not \in E_v
\hbox{ and } h^{a^{-1}x} \not \in E_v.
\end{equation}

Since $G_{(v_0,v_1,\ldots,v_{t})} = 1$, we see that $\beta \ge 1$ and since $h\in E_{v_0}=G_{v_0}$,
we see that $\beta\le t$. 
Recall that $G_{(v_0,v_1,\ldots,v_{t-i})} = \langle x_0,\ldots,x_{i-1}\rangle$
and that $e$ is the smallest integer such that $E=\langle x_0,\ldots,x_{t-1+e}\rangle$.
By definition, $h$ is an automorphism of $\vGa$ fixing the $(t-\beta)$-arc $(v_0,v_1,\ldots,v_{t-\beta})$ 
and moving the vertices $v_{-1}$ and $v_{t-\beta+1}$.
Since $v_0 = v$ and $E_{v}=G_{v}=\langle x_0,\ldots,x_{t-1}\rangle$, we may write 
$h=x_\alpha x_{\alpha+1}^{\varepsilon_{\alpha+1}}\cdots x_{\gamma-1}^{\varepsilon_{\gamma-1}} x_{\gamma}$, for some $0\le \alpha< \gamma< t$ and $\varepsilon_i\in \{0,1\}$.
Then $h\in \langle x_0, \ldots x_\gamma\rangle = G_{(v_0, \ldots, v_{t-\gamma-1})}$ and thus by
the definition of $\beta$ we see that $\gamma = \beta -1$.
Further, since $E_{v_{-1}} = E_{v^{a}} = (E_{v})^a = \langle x_1,\ldots,x_t\rangle$ and since $h\not\in E_{v_{-1}}$,
we see that $\alpha = 0$. Therefore
$h = x_0 x_{1}^{\varepsilon_{1}}\cdots x_{\beta-2}^{\varepsilon_{\beta-2}} x_{\beta-1}$
and thus
$$
h^{a^{t-\beta+j}x} = x^{-1}x_{t-\beta+j} x_{t-\beta+j+1}^{\varepsilon_1}\cdots x_{t+j-2}^{\varepsilon_{\beta-2}} x_{t+j-1} x.
$$
Suppose now that $h^{a^{t-\beta+j}x} \in E_v = \langle x_0,\ldots, x_{t-1}\rangle$. Since $E/N$ is an abelian group,  then
the above equality, when considered modulo the group $N$, implies 
$$
Nx_{t-\beta+j}\, Nx_{t-\beta+j+1}^{\varepsilon_1}\cdots Nx_{t+j-2}^{\varepsilon_{\beta-2}}\, Nx_{t+j-1} \in \langle Nx_0,\ldots, Nx_{t-1}\rangle.
$$
Since $j\ge 1$ and $t-\beta \ge 0$, we then see that
$$
Nx_{t+j-1} \in \langle Nx_0,\ldots, Nx_{t+j-2}\rangle.
$$
But since $j\le e$ this contradicts the fact that $\{Nx_0, Nx_1, \ldots, Nx_{t+e-1}\}$
is a minimal generating set for $E/N$; see (\ref{eq:E/N}). This contradiction shows that
$h^{a^{t-\beta+j}x} \not \in E_v$, as claimed. Similarly, if $h^{a^{-1}x} \in E_v$, then 
$h \in (E_v)^{x^{-1}a}$ and so $h\in x^a \langle x_1, \ldots, x_t \rangle (x^a)^{-1}$,
showing that $Nx_0 \in \langle Nx_1, \ldots, Nx_t \rangle$, again contradicting (\ref{eq:E/N}).
\smallskip

Now let $g$ be a non-trivial element of $G$ 
with $\fpr_{\V\Gamma}(g) > \frac{1}{3}$  and let $u\in \Fix_{\V\Gamma}(g)$.
 Since $g\in G_u = E_u$, we see that $[g,E]_u \le [E,E]_u = N_u=1$.
 In particular, since $[g,E_u] \le [g,E]_u$, it follows that $[g,E_u] = 1$,
 and thus ${\cent E g}_u = E_u$ for every $u\in \Fix_{\V\Gamma}(g)$.
We may now apply Lemma~\ref{lemma:1} with $\V\Gamma$ in place of $\Omega$, with 
$E$ in place of $X$ and with $u$ in place of $\omega$,
to conclude that
 \begin{equation}
 \label{eq:f1}
\Fix_{u^E}(g)=u^{\cent E g}\quad \hbox{and} \quad
\frac{1}{3} < \fpr_{\V\Gamma}(g) \le \fpr_{u^E}(g) =
\frac{1}{|E : \cent E g|}
 \end{equation}
and so $ |E:\cent E g| \le 2$. If $E=\cent E g$, then $g\in \Z E\cap E_v=(\Z E)_v=1$, a contradiction. Therefore $|E:\cent E g|=2$ and thus
 \begin{equation}
 \label{eq:f2}
\fpr_{u^E}(g)=
\begin{cases}
\frac{1}{|E:\cent E g|} = \frac{1}{2}&\mathrm{if\> }\Fix_{u^E}(g)\ne \emptyset,\\
0&\mathrm{if\> }\Fix_{u^E}(g)=\emptyset.
\end{cases}
 \end{equation}

Now assume without loss of generality  that $v\in \Fix_{\V\Gamma}(g)$.
Recall that $E=E_{t+e} = \langle x_0, \ldots, x_{t+e-1}\rangle$. Since $\Fix_{\V\Gamma}(g) \not = \emptyset$ it follows by
(\ref{eq:f2}) that $|E:\cent Eg|=2$ and thus there exists the smallest integer $i \in \{0,\ldots,t+e-1\}$ such that $x_i \not \in \cent Eg$.
 If $i< t$, then
  $x_i\in E_v = {\cent Eg}_v$
  (see \eqref{eq:Ei}), contradicting the choice of $i$.
 Hence we have
\begin{equation}
\label{eq:triglav}
E=\cent Eg \cup x_i\cent Eg \>\>\hbox{ for some }\> i\in \{t, t+1,\ldots,t+e-1\}.
\end{equation}

Since the automorphism $a$ maps a vertex to its neighbour in $\Gamma$, the connectivity of $\Gamma$ implies that
$G=\langle G_v, a\rangle \le \langle E,a\rangle$.
Now observe that
 $\vGa/E \cong (\vGa/N)/(E/N) \cong \vC(r,s) / ((G/N) \cap \K) \cong \vC_r$;
 in particular,
  the $E$-orbits on $\V\vGa$ can be labelled by $\Delta_i$, $i\in \ZZ_r$,
 in such a way that every arc of $\vGa$ starting in some $\Delta_i$ ends in $\Delta_{i+1}$. We may assume without loss of generality that
 $v\in \Delta_0$.
 Since the automorphism $a$ maps the vertex $v$ to its in-neighbour (see (\ref{eq:vi})), it follows that $(\Delta_j)^a = \Delta_{j-1}$
 for every $j\in \ZZ_r$, and thus 
 $$
 \Delta_j = v^{Ea^{-j}} = v^{a^{-j}E} = (v_{j})^E.
 $$
Using~\eqref{eq:triglav}, we can split this $E$-orbit into two halves, that is, 
$$
\Delta_j=\Delta_j' \cup \Delta_j'' \>\>\hbox{ where }\> \Delta_j' = v^{a^{-j}\cent E g},\>\>  \Delta_j'' = v^{a^{-j}x_i\cent Eg}\> \hbox { and } \>
\Delta_j' \cap \Delta_j'' = \emptyset.
$$

Let us call the $E$-orbit $\Delta_j$ {\em blue} provided that $\Fix_{\Delta_j}(g) = \emptyset$, {\em red} if $g \in E_{v_j}$, and
{\em pink} if $g \in E_{v_j^{x_i}}$.
By~\eqref{eq:f1} we see that $\Delta_j$ is red if and only if $\Fix_{\Delta_j}(g) = \Delta_j'$ 
and that it is pink if and only if $\Fix_{\Delta_j}(g) = \Delta_j''$. In particular, if $\Delta_j$ is not blue, then it is either red or pink.
Moreover, if $\Delta_j$ is red, then $g\in E_{v_j} = (E_v)^{a^{-j}}$, and if it is pink, then 
$g\in E_{(v_j)^{x_i}} = (E_v)^{a^{-j}x_i} = (E_v)^{x_{i+j}a^{-j}}$.
This immediately implies that:
\begin{equation}
\label{eq:riff}
\Delta_j \hbox{ is red }\> \Longleftrightarrow \> g^{a^{j}} \in E_v\quad \hbox{ and }\quad \Delta_j \hbox{ is pink }\> \Longleftrightarrow \> g^{a^{j}x_{i+j}}\in E_v.
\end{equation}

If, for a colour $X\in\{\hbox{red}, \hbox{ pink}, \hbox{ blue}\}$, an element $k\in \ZZ_r$ and a positive integer $\ell$  the orbits $\Delta_{k}, \ldots, \Delta_{k+\ell-1}$ are all of colour $X$
while $\Delta_{k-1}$ and $\Delta_{k+\ell}$ are of a colour different than $X$, we say that $S:=\{k,\ldots,k+\ell-1\}$ is a {\em strip of colour} $X$ and of length $\ell(S):=\ell$.
Let $S:=\{k,\ldots,k+\ell-1\}$ be a strip. Then the strip containing $k-1$ is said to {\em precede} $S$ and the strip
containing $k+\ell$ {\em follows} the strip $S$.

Let $S$ be a strip preceded by a strip $S^-$ and followed by a strip $S^+$.
We will now show that the following holds:
\begin{equation}
\label{eq:claim4}
\hbox{ If } S \hbox{ is red or pink, then } S^+ \hbox{ and } S^- \hbox{ are blue, } \ell(S) \le t \hbox{ and  } \ell(S^+)\ge e.
\end{equation}

Suppose first that $S=\{k,\ldots,k+\ell-1\}$ is a red strip. Let 
$h:= g^{a^k}$. 
Then, by (\ref{eq:riff}), we see that
$h^{a^{j}} \in E_v$ for every $j\in\{0,\ldots, \ell -1\}$, while $h^{a^{-1}}, h^{a^{\ell}} \not \in E_v$.
In other words, $h\in E_{v_0}\setminus E_{v_{-1}}$, $h\in E_{(v_0,\ldots,v_{\ell-1})}$ while $h\not \in E_{v_\ell}$.
Hence  $\beta := \min \{b \in \NN : h\in G_{(v_0, \ldots, v_{t-b})}\} = t -\ell + 1$. We may now apply (\ref{eq:tech})
to conclude that $1\le t-\ell+1\le t$ (implying $\ell \le t$, as required)
 and that neither of the elements $h^{a^{-1}x}$ and $h^{a^{\ell-1+j}x}$ for $j\in\{1,\ldots,e\}$ and $x\in E$
belongs to $E_v$. In view of (\ref{eq:riff}), this implies that neither of the orbits $\Delta_{k-1}$ and 
$\Delta_{k+\ell}, \ldots, \Delta_{k+\ell+e-1}$
are red or pink, showing that the strips $S^{-1}$ and $S^+$ are blue and that $\ell(S^+)\ge e$. This
completes the proof of the claim (\ref{eq:claim4}).

Since $e\ge t$ (see (\ref{eq:mygod1})), the claim (\ref{eq:claim4}) implies that the number of blue orbits
is greater of equal to the number of red and pink orbits combined. Since $g$ has no fixed points in blue orbits
and fixes precisely half of the points in each red or pink orbit, this shows that $g$ fixes at most $1/4$ of the vertices
of $\Gamma$. This contradictions completes the proof of the theorem in the case where
$\Gamma/N \cong \C(r,s)$ for some $r$ and $s$.
\end{proof}

We can now prove Theorem~\ref{thrm:1} under the assumption that $G$ is not $2$-arc-transitive.

\begin{theorem}
\label{the:GHAT}
Let $\Gamma$ be a connected $4$-valent graph, let $G$ be an edge- and vertex-transitive but not
$2$-arc-transitive group of automorphisms of $\Gamma$ and let $g$ be a non-trivial element of
$G$ with $\fpr_{\V\Gamma}(g) > 1/3$.
 Then $\Gamma\cong \C(r,s)$ for some positive integers $r$ and $s$ with $1\le s< 2r/3$.
\end{theorem}

\begin{proof}
Suppose that the theorem is false and let $\Gamma$ be a counterexample
with the smallest number of vertices.
Moreover, among all groups $G$ satisfying the assumptions of the theorem,
choose one of smallest order.

By Lemma~\ref{lem:GHAT}, there exists a minimal normal subgroup $N$ of $G$ of order
a power of $2$ acting semiregularly on $\V\Gamma$ such that $\Gamma/N$ is
a $4$-valent graph. Then $G/N$ acts edge- and vertex-transitively on $\Gamma/N$ but
not $2$-arc-transitively, and by Lemma~\ref{lem:quofix}, we see that $\fpr_{\V\Gamma}(Ng) > 1/3$.
The minimality of $\Gamma$ then implies that $\Gamma/N \cong \C(r',s')$ for some $r'$ and $s'$ with $1\le s'< 2r'/3$. By Lemma~\ref{lemma:technical}, it follows that $G/N$ is $\Aut(\Gamma/N)$-conjugate to a subgroup of $\rH$. Without loss of generality we may thus assume that 
$G/N \le \rH$. Furthermore, by Lemma~\ref{lem:rs}, we see that $Ng \in \K \le \rH^+$.
Now consider the group $X:=G/N \cap \rH^+$. Since $|\rH:\rH^+| = 2$, we see that $|G/N : X| \le 2$
and $X$ is a $\frac{1}{2}$-arc-transitive group of automorphisms of $\Gamma/N$.
Let $G^+$ be the preimage of $X$ with respect to the quotient projection $G \to G/N$.
Then $G^+/N \cong X \le \rH^+$, $G^+$ is $\frac{1}{2}$-arc-transitive and since $Ng\in X$,
we see that $g\in G^+$. By our choice of $G$ this implies that $G=G^+$,
and hence $G/N \le \rH^+$. The result now follows from Lemma~\ref{lem:vCcov}.
\end{proof}
\color{black}

\section{Proof of Theorem~\ref{thrm:2} for $\Gamma$ not arc-transitive}

We now move our attention to $3$-valent vertex- but not arc-transitive graphs.
As observed in~\cite{PSV,census1}, the $3$-valent graph admitting a vertex- but not arc-transitive group of automorphisms are closely related to the family of $4$-valent graph admitting an arc- but not $2$-arc-transitive group of automorphisms.
This will enable us to reduce the proof of Theorem~\ref{the:cubicNonAT} below to the situation covered by Theorem~\ref{thrm:1}.

In the proof of Theorem~\ref{the:cubicNonAT} we need to refer to two special families of cubic vertex-transitive graphs: the {\em prisms} $\Prism_n$, that can be defined as the Cayley graphs
$\Cay(\ZZ_n\times\ZZ_2,\{(0,1),(1,0),(-1,0)\})$ for $n\ge 3$,
 and the {\em M\"{o}bius ladders} $\Mob_n$, defined as the Cayley graphs
 $\Cay(\ZZ_{2n},\{1,-1,n\})$ for some $n\geq 2$.

\begin{theorem}
\label{the:cubicNonAT}
Let $\Gamma$ be a connected $3$-valent graph admitting a vertex-transitive but not arc-transitive group of automorphisms $G$. Let $g \in G$ be a non-trivial element of smallest order such that
 $\fpr_{\V\Gamma}(g) > 1/3$. Then $\Gamma$ is either a Split Praeger-Xu graph
 $\Spl(\C(r,s))$ with $1\le s \le 2r/3$, $r\ge 3$, or
 isomorphic to $\Lambda_1$ (the complete graph $\K_4$) or $\Lambda_3$ (the skeleton of the cube).
\end{theorem}

\begin{proof}
 By consulting
the database \cite{census1} of $3$-valent vertex-transitive graphs on at most $1280$ vertices,
we checked that Theorem~\ref{thrm:2} holds if $|\V\Gamma|\le 1280$.
We may thus assume that $|\V\Gamma|>1280$ (in fact, we will only use
$|\V\Gamma|>140$).
 Observe the vertex-stabiliser $G_v$ is 
 a $2$-group (and thus the order $o(g)$ of $g$ is $2$)
 whose action upon $\Gamma(v)$ has two orbits, one of length $2$ and
 one of length $1$.
  
For a vertex $w \in \V(\Gamma)$ let $w'$ be the neighbour 
of $w$ such that $\{w\}$ is the orbit of $G_w$ of length $1$.
Then clearly $w'' = w$ and $G_w= G_w'$. Hence, the set $\cM:=\{\{w,w'\}: w\in \V\Gamma\}$ is  a complete matching of $\Gamma$, while edges outside $\cM$ form
a $2$-factor $\cF$. The group $G$ preserves both $\cF$ and $\cM$ and acts
 transitively on the arcs of each of these two sets. Let
 ${\tilde{\Gamma}}$ be the graph with vertex-set $\cM$ and two vertices
 $e_1 ,e_2 \in \cM$ adjacent if and only if they are (as edges of $\Gamma$)
 at distance $1$ in $\Gamma$. The graph ${\tilde{\Gamma}}$ is then called the {\em merge} of 
 $\Gamma$. We may also think of $\Gamma$ as being obtained by contracting all the edges in $\cM$.  The group $G$ clearly acts as an arc-transitive group of automorphisms on ${\tilde{\Gamma}}$.
 Moreover, the connected components of the the $2$-factor $\cF$ gives rise to a decomposition $\cC$ of $\E{\tilde{\Gamma}}$ into cycles.
 
 If $\Gamma \cong \Prism_n$ or $\Mob_n$  for some $n\ge 3$, then it is easy to see that a non-trivial automorphism of $\Gamma$ can fix at most $4$ vertices, which, together with the assumption $\fpr_{\V\Gamma}(g) > 1/3$ implies that $|\V\Gamma| < 12$, contradicting our assumption on $\Gamma$.
 We may thus assume that $\Gamma$ is neither a prism nor a M\"{o}bius ladder.
 As was shown in \cite[Lemma 9 and Theorem 10]{census1}, this implies that ${\tilde{\Gamma}}$ is $4$-valent.
 Moreover,  the action of $G$ on $\V{\tilde{\Gamma}}$ is faithful, arc-transitive but not 2-arc-transitive.
Observe also that $\fpr_{\V{\tilde{\Gamma}}}(g) \ge \fpr_{\V\Gamma}(g)>1/3$.
By Theorem~\ref{thrm:1}, it thus follows that ${\tilde{\Gamma}} \cong \C(r,s)$ with
$1\le s<2r/3$, $r\ge 3$. In view of \cite[Theorem 12]{census1}, the graph
$\Gamma$ can then be uniquely reconstructed from ${\tilde{\Gamma}}$ and the decomposition $\cC$ of $\E{\tilde{\Gamma}}$ arising from the $2$-factor $\cF$ via the {\em splitting operation} defined in
\cite[Construction 11]{census1}. In short, $\Gamma$ can be obtained from ${\tilde{\Gamma}}$ by
splitting each vertex $v$ of ${\tilde{\Gamma}}$ into two adjacent vertices $v', v''$, each of them retaining
two neighbours of $v$ in ${\tilde{\Gamma}}$, that together with $v$ form a part of a cycle in $\cC$.
It is then straightforward to see that $\Gamma$ is the Split Praeger-Xu graph $\Spl(\C(r,s))$;
or, which is equivalent, that the merging operation applied to $\Spl(\C(r,s))$ yields the
 graph $\C(r,s)$. 
\end{proof}

\color{black}

\section{Graph-theoretical consideration}
\label{sec:geometry}

In this section we make a digression into purely graph-theoretical considerations.
We begin with an easy observation about $3$-valent vertex-transitive graphs,
and then prove the in the $4$-valent arc-transitive case we may assume that
a non-trivial element fixing more than $1/3$ vertices fixes an arc of the graph.

\begin{lemma}
\label{lem:unworthy3}
Let $\Gamma$ be a connected $3$-valent vertex-transitive graph. If there exist two distinct
vertices $u$ and $u'$ of $\Gamma$ such that $\Gamma(u) = \Gamma(u')$, then
$\Gamma\cong \K_{3,3}$.
\end{lemma}

\begin{proof}
Let $\Gamma(u) =\Gamma(u') = \{v_1,v_2,v_3\}$. Since $\Gamma$ is vertex-transitive,
there exist $v_1'\in \V\Gamma$ such that $\Gamma(v_1) = \Gamma(v_1')$.
But then $v_1' \in \Gamma(u)$, implying that $v_1'$ is one of the vertices $v_2$ or $v_3$,
say $v_1' = v_2$.
By applying the same argument to $v_3$ in place of $v_1$, we see that $\Gamma(v_1) = \Gamma(v_2) = \Gamma(v_3)$. But then connectivity of $\Gamma$ yields
$\Gamma\cong \K_{3,3}$.
\end{proof}

\begin{theorem}\label{geometry}
Let $k\in \{3,4\}$ and let $\Gamma$ be a connected $k$-valent arc-transitive graph admitting
a non-trivial automorphism $g$  fixing no arc of $\Gamma$ and satisfying $\fpr_{\V\Gamma}(g) >1/3$.
Then $k=4$ and $\Gamma\cong \C(r,1)$ for some positive integer $r$, $r\ge 3$,
or $k=3$ and $\Gamma \cong \K_{3,3}$.
\end{theorem}
\begin{proof}

Let us first consider the case $k=3$. Let $d$ be the minimal distance between two
 vertices fixed by $g$. Since $g$ fixes no arcs of $\Gamma$, we see that $d\ge 2$.
 If $d\ge 3$, then 
 every vertex $v$ in $F':=\V\Gamma \setminus \Fix_{\V\Gamma}(g)$ is adjacent to at most one vertex in $F:=\Fix_{\V\Gamma}(g)$, while every vertex $u\in F$ is adjacent to three vertices in $F'$. Therefore, $|F'| \ge 3 |F|$ and thus 
 $$\fpr_{\V\Gamma} = \frac{|F|}{|F| + |F'|} \le \frac{|F|}{|F| + 3|F|} = \frac{1}{4},$$
 a contradiction. Hence $d=2$. Let $v$ and $w$ be two vertices at distance $2$ fixed by $g$.
 If $\Gamma(v) = \Gamma(w)$, then by Lemma~\ref{lem:unworthy3}, $\Gamma\cong\K_{3,3}$.
 If $|\Gamma(v) \cap \Gamma(w)| =1$, then the vertex in $\Gamma(v) \cap \Gamma(w)$ is also
 fixed by $g$, contradicting $d=2$. Therefore $\Gamma(v) \cap \Gamma(w) = \{u_1,u_2\}$ with
 $u_1\not = u_2$. But then $g$ fixes the vertex in $\Gamma(v) \setminus \{u_1,u_2\}$,
 contradicting $d=2$. This complete the proof in the case $k=3$.

Let us now assume that $k=4$.
 We divide the proof into several steps. We start by recalling that a connected 
 $4$-valent arc-transitive graph containing two distinct vertices $w$ and $w'$ with $\Gamma(w)=\Gamma(w')$ is isomorphic to $\C(r,1)$ for some $r\ge 3$; for the proof, see \cite[Lemma~$4.3$]{Pwilson}, for instance.
 For the rest of the argument, we may thus assume that $\Gamma$ has no two distinct vertices with the same neighbourhood. 
\smallskip

\noindent\textsc{\bf Step~1: }For every four distinct vertices $v_1,v_2,v_3,v_4\in \Fix_{\V\Gamma}(g)$, we have $\Gamma(v_1)\cap\Gamma(v_2)\cap \Gamma(v_3)\cap \Gamma(v_4)= \emptyset$.
\medskip

\noindent We argue by contradiction and we suppose that there exist four distinct vertices $v_1,v_2,v_3,v_4\in\Fix_{\V\Gamma}(g)$ with  $\Gamma(v_1)\cap\Gamma(v_2)\cap \Gamma(v_3)\cap \Gamma(v_4)\ne \emptyset$. Let $w\in \Gamma(v_1)\cap\Gamma(v_2)\cap \Gamma(v_3)\cap \Gamma(v_4)$. Observe that $w^g\in \Gamma(v_1)\cap\Gamma(v_2)\cap \Gamma(v_3)\cap \Gamma(v_4)$ because $v_1,v_2,v_3,v_4$ are fixed by $g$, and $w^g\ne w$ because $g$ fixes no arc of $\Gamma$. Thus $\Gamma(w)=\{v_1,v_2,v_3,v_4\}=\Gamma(w^g)$, which is a contradiction.  
\medskip

\noindent\textsc{\bf Step~2: }For every three distinct vertices $v_1,v_2,v_3\in \Fix_{\V\Gamma}(g)$, we have $\Gamma(v_1)\cap\Gamma(v_2)\cap \Gamma(v_3)= \emptyset$.
\smallskip

\noindent We argue by contradiction and we suppose that there exist three distinct vertices $v_1,v_2,v_3\in \Fix_{\V\Gamma}(g)$ with $\Gamma(v_1)\cap\Gamma(v_2)\cap \Gamma(v_3)\ne \emptyset$.
Let $w\in \Gamma(v_1)\cap\Gamma(v_2)\cap \Gamma(v_3)$. Arguing as in \textsc{Step 1}, $w^g\in \Gamma(v_1)\cap\Gamma(v_2)\cap \Gamma(v_3)$ and $w^g\ne w$. Thus 
\begin{equation}
\label{eq:qe1}
w\,\, \textrm{and}\,\, w^g \textrm{ have three neighbours in common}.
\end{equation}

From this point onwards one could follow the proof of the Subcase II.A of  \cite[Theorem 3.3]{Pwilson}
to conclude that then $\Gamma \cong \K_{5,5} - 5\K_2$ (yielding a contradiction).
However, for the sake of completeness, we provide an independent proof of Step 2.

If $w$ is adjacent to $w^g$ in $\Gamma$, then from the arc-transitivity of $\Gamma$ we deduce $\Gamma$ is isomorphic to the complete graph $K_5$. Since $g$ fixes no arc of $\Gamma$, we have $|\Fix_{\V\Gamma}(g)|\le 1$ and hence $\fpr_{\V\Gamma}(g)\le 1/5<1/3$ and~\eqref{v3} holds. Thus, we may suppose for the rest of the proof of this step that $w$ is not adjacent to  $w^g$.

Let us now prove that $w^{g^2} = w$. If that were not the case, then  $w,w^g$ and $w^{g^2}$ are all adjacent to $v_1, v_2$ and $v_3$. Moreover, since $v_1$ and $v_2$
cannot have all neighbours in common, we also see that $w^{g^3} = w$.
Let $u_1$ be the fourth neighbour of $v_1$ other than $w,w^g$ and $w^{g^2}$. Since $g$ fixes no arcs, $u_1^g \not = u_1$, and hence
$u_1^g$, being adjacent to $v_1$, is one of  $w=w^{g^3},w^g$ and $w^{g^2}$. But then $u_1 \in
\{w^{g^2},w,w^g\}$, yielding a contradiction. This shows that $w^{g^2} = w$, as claimed.

Let $v_4\in \V\Gamma$ with $\Gamma(w)=\{v_1,v_2,v_3,v_4\}$. If $v_4\in \Fix_{\V\Gamma}(g)$, then $w\in\Gamma(v_1)\cap\Gamma(v_2)\cap \Gamma(v_3)\cap\Gamma(v_4)$ and $\Gamma(w)=\{v_1,v_2,v_3,v_4\}=\Gamma(w^g)$, that is, $w$ and $w^g$ are two distinct vertices with the same neighbourhood, contradicting our assumption. Therefore $v_4$ is not fixed by $g$. Thus $\Gamma(w^g)=\{v_1,v_2,v_3,v_4^g\}$ and $v_4\ne v_4^g$.
Note that since $w^{g^2} = w$, we have $v_4^{g^2} = v_4$.
 For the next two paragraphs Figure~\ref{fig1} might be of some help for following the argument.
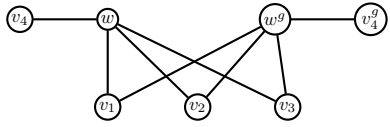
\begin{figure}[!htb]
\begin{center}
\begin{tikzpicture}[node distance=1.2cm,thick,scale=0.7,every node/.style={transform shape}]
\node[circle,draw,inner sep=1pt](v4){$v_4$};
\node[right=of v4,circle,draw, inner sep=1pt](w){$w$};
\node[below=of w,circle,draw, inner sep=1pt](v1){$v_1$};
\node[right=of v1,circle,draw, inner sep=1pt](v2){$v_2$};
\node[right=of v2,circle,draw, inner sep=1pt](v3){$v_3$};
\node[right=of w](d){};
\node[right=of d,circle,draw, inner sep=1pt](wg){$w^g$};
\node[right=of wg,circle,draw, inner sep=1pt](v4g){$v_4^g$};
\draw(v4) to (w);
\draw(w) to (v1);
\draw(w) to (v2);
\draw(w) to (v3);
\draw(wg) to (v1);
\draw(wg) to (v2);
\draw(wg) to (v3);
\draw(wg) to (v4g);
\end{tikzpicture}
\end{center}
\caption{Graph for the proof of Theorem~\ref{geometry}, I} \label{fig1}
\end{figure}
Since $\Gamma$ is vertex-transitive,~\eqref{eq:qe1} yields that for each of $v_i$, $i\in\{1,2,3\}$, there exists $v_i'\in \V\Gamma$ with $v_i$ and $v_i'$ having three neighbours in common. 

Our next claim is that that each of $v_i$, $i\in\{1,2,3\}$, has three neighbours in common
with $v_4$ and three neighbours in common with $v_4^g$. Due to the symmetry conditions,
it suffice to show that $v_1$ has three neighbours in common with $v_4$.

Since $w,w^g\in \Gamma(v_i)$, by the pigeonhole principle, either $w$ or $w^g$ is a common neighbour of $v_1$ and $v_1'$. Without loss of generality, we may assume that $w\in \Gamma(v_1)\cap\Gamma(v_1')$. As $\Gamma(w)=\{v_1,v_2,v_3,v_4\}$, we deduce $v_1'\in \{v_2,v_3,v_4\}$.
 
We first suppose that $v_1'\in\{v_2,v_3\}$.
Without loss of generality, we may assume that $v_1'=v_2$. Let us call $v_5$ the third vertex in common to $v_1$ and $v_2$. Clearly, $v_5$ cannot be fixed by $g$, otherwise $g$ fixes the arc $(v_1,v_5)$. Since $g$ fixes $v_1$ and $v_2$, we obtain that $v_5^g$ is a neighbour of both $v_1^g=v_1$ and $v_2^g=v_2$. Thus $\Gamma(v_1)=\{w,w^g,v_5,v_5^g\}=\Gamma(v_2)$, contradicting the fact that $\Gamma$ has no two distinct vertices with the same neighbourhood. This paragraph shows that 
$v_1' \not \in \{v_2, v_3\}$ and hence $v_1'=v_4$.
Since $v_1$ has three neighbours in common with $v_4$, we deduce that $v_1^g=v_1$ has three neighbours in common with $v_4^g$.

By symmetry, the argument in the previous four paragraphs can be applied also to the vertex $v_2$ and $v_3$. Therefore, we deduce that each of $v_1$, $v_2$ and $v_3$ has  three neighbours in common with $v_4$ and three neighbours in common with $v_4^g$.

Since $v_1$ has three neighbours in common with $v_4$ and three neighbours in common with $v_4^g$, we deduce that $v_1$, $v_4$ and $v_4^g$ must have at least two neighbours in common. These vertices cannot be $w$ or $w^g$, otherwise we contradicting Figure~\ref{fig1}. Thus, let us call $v_5$ one of the two neighbours in common to $v_1$, $v_4$ and $v_4^g$. As $g$ fixes no arcs, we have $v_5^g\ne v_5$. Thus $v_5^g$ is a neighbour in common to $v_1^g=v_1$, $v_4^g$ and $(v_4^g)^g=v_4$. The left side of Figure~\ref{fig2} might be of some help for following the rest of the argument.
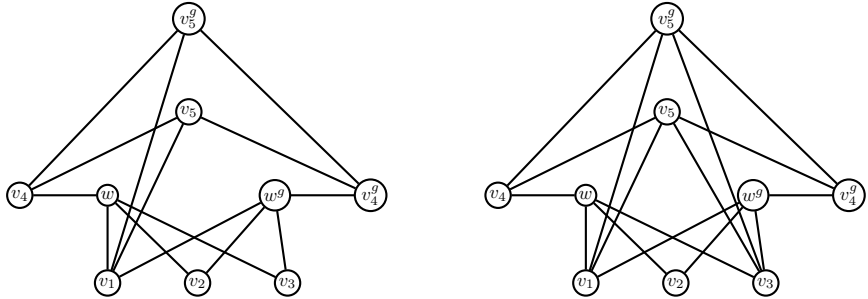
\begin{figure}[!htb]
\begin{center}
\begin{tikzpicture}[node distance=1.2cm,thick,scale=0.7,every node/.style={transform shape}]
\node[circle,draw,inner sep=1pt](v4){$v_4$};
\node[right=of v4,circle,draw, inner sep=1pt](w){$w$};
\node[below=of w,circle,draw, inner sep=1pt](v1){$v_1$};
\node[right=of v1,circle,draw, inner sep=1pt](v2){$v_2$};
\node[right=of v2,circle,draw, inner sep=1pt](v3){$v_3$};
\node[right=of w](d){};
\node[right=of d,circle,draw, inner sep=1pt](wg){$w^g$};
\node[right=of wg,circle,draw, inner sep=1pt](v4g){$v_4^g$};
\node[above=of d,circle,draw, inner sep=1pt](v5){$v_5$};
\node[above=of v5,circle,draw, inner sep=1pt](v5g){$v_5^g$};
\draw(v4) to (w);
\draw(w) to (v1);
\draw(w) to (v2);
\draw(w) to (v3);
\draw(wg) to (v1);
\draw(wg) to (v2);
\draw(wg) to (v3);
\draw(wg) to (v4g);
\draw(v1) to (v5);
\draw(v1) to (v5g);
\draw(v4) to (v5);
\draw(v4) to (v5g);
\draw(v4g) to (v5);
\draw(v4g) to (v5g);
\end{tikzpicture}
\phantom{AAAA}
\begin{tikzpicture}[node distance=1.2cm,thick,scale=0.7,every node/.style={transform shape}]
\node[circle,draw,inner sep=1pt](v4){$v_4$};
\node[right=of v4,circle,draw, inner sep=1pt](w){$w$};
\node[below=of w,circle,draw, inner sep=1pt](v1){$v_1$};
\node[right=of v1,circle,draw, inner sep=1pt](v2){$v_2$};
\node[right=of v2,circle,draw, inner sep=1pt](v3){$v_3$};
\node[right=of w](d){};
\node[right=of d,circle,draw, inner sep=1pt](wg){$w^g$};
\node[right=of wg,circle,draw, inner sep=1pt](v4g){$v_4^g$};
\node[above=of d,circle,draw, inner sep=1pt](v5){$v_5$};
\node[above=of v5,circle,draw, inner sep=1pt](v5g){$v_5^g$};
\draw(v4) to (w);
\draw(w) to (v1);
\draw(w) to (v2);
\draw(w) to (v3);
\draw(wg) to (v1);
\draw(wg) to (v2);
\draw(wg) to (v3);
\draw(wg) to (v4g);
\draw(v1) to (v5);
\draw(v1) to (v5g);
\draw(v4) to (v5);
\draw(v4) to (v5g);
\draw(v4g) to (v5);
\draw(v4g) to (v5g);
\draw(v3) to (v5);
\draw(v3) to (v5g);
\end{tikzpicture}
\end{center}
\caption{Graphs for the proof of Theorem~\ref{geometry}, II} \label{fig2}
\end{figure}
As $v_3$ has three neighbours in common with $v_4$ and three neighbours in common with $v_4^g$, we may apply the argument of the previous paragraph with the vertex $v_1$ replaced by $v_3$.  We deduce that $v_3$, $v_4$ and $v_4^g$ have at least two neighbours in common, which cannot be neither $w$ nor $w^g$. From the graph of the left side of Figure~\ref{fig2}, we see that these mutually common neighbours are $v_5$ and $v_5^g$, otherwise we contradict the fact that $v_4$ and $v_4^g$ have valency $4$. See now the graph on the right side of Figure~\ref{fig2}.
When we apply the argument in the previous paragraph to the vertex $v_3$ replaced by the vertex $v_2$, we deduce that $v_5$ and $v_5^g$ are neighbours of $v_2$, contradicting the fact that $\Gamma$ has valency $4$.
\medskip

\noindent\textsc{\bf Step~3: }$\fpr_{\V\Gamma}(g)\le 1/3$.
\smallskip

\noindent For simplicity, set $F:=\Fix_{\V\Gamma}(g)$ and $F':=\V\Gamma\setminus\Fix_{\V\Gamma}(g)$. Since $g$ fixes no arc of $\Gamma$, for every $v\in F$, we have $\Gamma(v)\subseteq  F'$. Moreover, from \textsc{Step~2}, we see that, for every $v\in F'$, we have $|\Gamma(v)\cap F|\le 2$. Thus, 
by counting the edges between $F$ and $F'$, we obtain
 $4|F|\le 2|F'|$. As $|F|+|F'|=|\V\Gamma|$, it follows $$\fpr_{\V\Gamma}(g)=\frac{|F|}{|F|+|F'|}\le \frac{|F|}{|F|+2|F|}\le \frac{1}{3},$$
 which contradicts our assumptions.
 \end{proof}

\color{black}
\section{The $2$-arc-transitive case}
\label{sec:reduction}

In this section we complete the proofs of Theorems~\ref{thrm:1} and \ref{thrm:2},
by considering the remaining cases of $4$-valent $2$-arc-transitive graphs and $3$-valent
arc-transitive graphs.
These cases are considered in \cite{LehPotSpi} in a more general context 
of arc-transitive locally quasiprimitive graphs
(that is, graphs, where the stabiliser of a vertex acts quasiprimitively on the neighbourhood)
for which the order of the vertex-stabiliser is bounded by some constant depending only on the valence of the graph. There it is proved that
 for every constant $c$ there are only finitely 
 many graphs in such a family that
  admit a non-trivial automorphism fixing more than $1/c$ vertices.
Since the order of $\Aut(\Gamma)_v$ is bounded by $11\, 664$ if $\Gamma$ is a connected $2$-arc-transitive $4$-valent graph \cite{Weiss}, and by
$48$ if $\Gamma$ is 
a connected arc-transitive $3$-valent graph \cite{tutte},
the result proved in \cite{LehPotSpi} implies that there can be
only a finite number of counterexamples to Theorems~\ref{thrm:1} and \ref{thrm:2}
(however, with the bound on their order being too large to be practical).
The analysis carried out in this section is thus aimed at a finite number of graphs only.

We first prove two reduction results simultaneously for the $4$-valent and the $3$-valent case
and split our analysis later.
Note that by Theorem~\ref{geometry}, the element $g$ fixing more than $\frac{1}{3}$ vertices
fixes an arc. Moreover, if $\Gamma$ is a connected $3$-valent graph with $G\le \Aut(\Gamma)$ acting arc-transitive but not $2$-arc-transitively, then the arc-stabiliser $G_{vw}$ is trivial.
This shows that in our analysis we may assume that the graph $\Gamma$ is $2$-arc-transitive
not only in the $4$-valent case but also when $\Gamma$ is $3$-valent case.
If a group of automorphisms $G$ of a graph $\Gamma$ acts transitively on the $2$-arcs of $\Gamma$, we say that $\Gamma$ is $(G,2)$-arc-transitive.

\begin{lemma}
\label{lem:Nnonab}
Let $\Gamma$ be a connected $k$-valent $(G,2)$-arc-transitive 
graph with $k\in\{3,4\}$ and  
let $g$ be a nontrivial element of $G$ with $\fpr_{\V\Gamma}(g) > 1/3$.
Suppose that $G$ contains a minimal normal subgroup $N$ 
such that $\Gamma/N$ is isomorphic to 
one of the graphs $\Psi_1, \ldots, \Psi_6, \C(r,s)$, $1\le s \le 2r/3$, $r\ge 3$ (if $k=4)$;
or to one of the graphs $\Lambda_1,\ldots,\Lambda_6$ (if $k=3$).
Then $\Gamma$ itself is isomorphic to one of these graphs and is therefore
not a counterexample to Theorem~\ref{thrm:1} or Theorem~\ref{thrm:2}.
\end{lemma}

\begin{proof}
If $|\V\Gamma|\le 768$ and $k=4$ or if $|\V\Gamma| \le 10\,000$ and $k=3$,  the claim can be checked with a computer assisted computation using the census of connected $4$-valent 
$2$-arc-transitive graphs of order at most $768$ \cite{Primoz} and the
census of connected $3$-valent arc-transitive graphs of order at most $10\,000$ 
\cite{condercensus}.
We may therefore assume that $|\V\Gamma|$ exceeds these bounds.

Since $\Gamma/N$ is of the same valency as $\Gamma$, it follows that $N_v =1$ for every $v\in\V\Gamma$ and that $G/N$ acts faithfully on $\V\Gamma/N$. In particular, $Ng\in G/N$
is a non-trivial automorphism of $\Gamma/N$ fixing more than $1/3$ of the vertices.
Furthermore, $G/N$ acts transitively on the arcs of $\Gamma/N$ and if $k=4$ then it also acts transitively on the $2$-arcs of $\Gamma/N$. Consequently, if $\Gamma/N \cong \C(r,s)$, then by Remark~\ref{rem:PX2AT}, $r=4$ and $s\in \{1,2\}$.
By applying Lemma~\ref{lemma:1}
 with $N$ in place of $X$ and $\V\Gamma$ in place of $\Omega$, we
 conclude that $1/3 < \fpr_{\V\Gamma}(g) \le 1/|N:\cent N g|$, implying that
   $|N:\cent N g|\le 2$. 

Suppose first that $N$ is an elementary abelian $2$- or $3$-group. Since $N$ is a minimal normal subgroup of $G$, the action of $G/N$ on $N$ by conjugation endows $N$ with the structure of a $G/N$-irreducible module over a field $\mathbb{F}$ of size $2$ or $3$. 
In this case the proof can be completed by a straightforward computation with the computer algebra system, such as {\sc Magma}~\cite{magma}, 
in a way which we now describe.

For every graph $\Delta \in \{\Psi_1,\ldots, \Psi_6, \C(r,1),\C(r,2),\Lambda_1, \ldots, \Lambda_6\}$
we consider every $2$-arc-transitive subgroup $H$ of $\Aut(\Delta)$
 and contains a non-trivial element $h$ fixing more than $1/3$ of the vertices of $\Delta$. Thus $(\Delta,H,h)$ is our putative triple $(\Gamma/N,G/N, Ng)$.

 Next, we compute all the irreducible $\mathbb{F}H$-modules $V$. Since $|N:\cent N g|\le 2$,
the element $g$ either centralises $N$, or $p=2$ and $g$ acts as a transvection on $N$. Among all irreducible $\mathbb{F}H$-modules $V$, we select those with $\cent HV\ne 0$ or (in the case $p=2$) those admitting an element $h$ of $H$ with $|V:\cent V h|=2$. Thus, in this refined family, $V$ is our putative $N$. 

 When $k=3$, a direct computation shows that all such modules $V$ satisfy
 $|\V\Gamma/N|\cdot p^{\dim V} \le 10\,000$, contradicting our assumption that
$|\V\Gamma| > 10\,000$. 
   
 In the case $(k,p)=(4,2)$ we have checked that
\begin{itemize}
\item $|\V\Gamma/N|\cdot 2^{\dim V}\le 640$, or
\item $\Gamma/N\cong \Psi_5$,  $H\cong \Sym(7)$, $\dim_{\mathbb{F}_2} (V)=6$ and there is only one choice for $V$, or
\item  $\Gamma/N\cong \Psi_6$, $H\cong \Sym(7)\times \C_2$, $\dim_{\mathbb{F}_2}(V)=6$ and there is only one choice for $V$. 
\end{itemize}
Since $|\V\Gamma/N||N|=|\V\Gamma|>640$, we may consider only the last two possibilities. For these cases, we have computed the cohomology module of $H$ over $V$ and we have obtained the corresponding first and second cohomology groups. These groups have dimension zero and hence $G$ splits over $N$ and $N$ has a unique conjugacy class of complements in $G$. Thus $G$ is isomorphic to a subgroup of $\mathbb{F}_2^6\rtimes \Sym(7)$ when $\Gamma/N\cong \Psi_5$ and $G$ is isomorphic to a subgroup of $\mathbb{F}_2^6\rtimes (\Sym(7)\times \C_2)$ when $\Gamma/N\cong \Psi_6$. In these cases, we have constructed the abstract group $G$ and we have considered all the permutation representations of $G$ of the relevant degree (of degree $2^6\cdot 35$ when $\Gamma/N\cong \Psi_5$ and of degree $2^6\cdot 70$ when $\Gamma/N\cong \Psi_6$). Finally, we have checked that none of these permutation groups acts arc-transitively on a connected $4$-valent graph.

In the case $(k,p)=(4,3)$ we know that $g$ centralises $N$, and hence we may consider only those
$\mathbb{F}H$-modules $V$ with $\cent HV\ne 0$. The computation in this case is similar to the
case $p=2$, and again none of the modules $V$ yields an appropriate group $G$.

We may thus assume that $N$ is not an elementary abelian $2$- or $3$-group. Since $|N:\cent N g|\le 2$ and since $N$ has no index-$2$ subgroups in this case, we deduce $g\in \cent GN$ and hence $C:=\cent G N$ is a normal subgroup of $G$ not acting semiregularly on $\V\Gamma$.  

Suppose  $v^{N}\subseteq v^{C}$. Then, for every $n\in N$, there exists $c\in C$ with $v^{nc}=v$, that is, $nc\in G_v$. Since $n$ and $c$ commute, the order $o(nc)$ of $nc$ equals $\lcm\{o(n),o(c)\}$.
Since $G_v$ is a $\{2,3\}$-group, we thus see that $o(nc)$
is a power of $2$ times a power of $3$. Thus $N$ is a $\{2,3\}$-group. From Burnside's $p^\alpha q^\beta$-theorem,  $N$  is solvable and hence 
elementary abelian, contradicting our assumption.

We may thus assume that  $v^N\nsubseteq v^C$. Observe that $G_v^{\Gamma(v)}$ is a primitive group, implying that $C_v^{\Gamma(v)}$ is either transitive or trivial. In the latter case, it follows that $C_v =1$ contradicting the fact that $g\in C_v$. Hence $C_v$ acts transitively on
$\Gamma(v)$, implying that
 $C$ is either transitive on $\V\Gamma$, or $\Gamma$ is bipartite with bipartition given by the orbits of $C$ on $\V\Gamma$. As $v^N\nsubseteq v^C$, we have $v^C\ne \V\Gamma$ and hence $C$ is not transitive on $\V\Gamma$; thus $\Gamma$ is bipartite with bipartition given by the $C$-orbits. As $v^N\nsubseteq v^C$,  $N$ contains permutations interchanging the two parts of the bipartition of $\Gamma$. Thus $N$ contains a subgroup having index $2$, which is a contradiction because $N$ is not a $2$-group.  
\end{proof}
\color{black}

Let us now assume that Theorem~\ref{thrm:1} or Theorem~\ref{thrm:2}
fails due to a $4$-valent $2$-arc-transitive graph or due to a $3$-valent arc-transitive graph,
respectively. Let us consider the minimal counterexample, that is, let us work under the following assumption:

\begin{hypothesis}
\label{hyp:sAT}
{\rm
Let $k\in\{3,4\}$ and
let  $\Gamma$ be a smallest
connected $k$-valent $2$-arc-transitive graph not isomorphic to any of the
exceptional graphs $\Psi_1, \ldots, \Psi_6, \Lambda_1, \ldots, \Lambda_6$ or $\C(r,s)$ with $1\le s \le 2r/3$, $r\ge 3$, 
but admitting a non-trivial automorphism fixing more than $1/3$ of the vertices.
Among such automorphisms, pick one of smallest order.
In view of Theorem~\ref{geometry}, it follows that such an automorphism
 fixes an arc $(v,w)$ of $\Gamma$.
Let $G$ be a smallest $2$-arc-transitive subgroup of $\Aut(\Gamma)$ containing $g$.
Since $G_{vw}$ is a $2$-group if $k=3$, and is
a $\{2,3\}$-group if $k=4$, we see that the order $o(g)$ of $g$ satisfies
$o(g) \in \{2,3\}$ if $k=4$ and $o(g) = 2$ if $k=3$. Since the validity of Theorems~\ref{thrm:1} and \ref{thrm:2} was checked for the graphs in the census of $4$-valent $2$-arc-transitive graphs
of order at most $768$ \cite{Primoz} and  the census of $3$-valent arc-transitive graphs
of order at most $10\,000$, we assume that $|\V\Gamma| > 768$ if $k=4$ and $|\V\Gamma| > 10\,000$ if $k=3$.}
\end{hypothesis}

\begin{lemma}
\label{lem:biquasi}
Assuming Hypothesis~\ref{hyp:sAT}, it follows that $G$ has a unique normal subgroup,
which is non-abelian, has at most $2$ orbits on $\V(\Gamma)$ and does not
act semiregularly on $\V\Gamma$.
\end{lemma}

\begin{proof}
Suppose that $G$ contains a minimal normal subgroup $N$ having at least $3$ orbits on $\V\Gamma$.
By \cite[Theorem 4.1]{11_1} it then follows that $N$ is semiregular,
and $\Gamma/N$ is $4$-valent with $G/N$
acting faithfully as as $s$-arc-transitive group of automorphisms.
By Lemma~\ref{lem:quofix}, $Ng\in G/N$ is a non-trivial automorphism of $\Gamma/N$
with $\fpr_{\V\Gamma/N}(Ng) > 1/3$. The minimality of $\Gamma$ now implies that
$\Gamma/N$ is one of the exceptional graphs $\Psi_1, \ldots, \Psi_6,\Lambda_1,\ldots,\Lambda_6$ or $\C(r,s)$ for some $r$ in $s$. But then, by Lemma~\ref{lem:Nnonab},
$\Gamma$ is not a counterexample to Theorem~\ref{thrm:1} or Theorem~\ref{thrm:2},
contradicting Hypothesis~\ref{hyp:sAT}.
We have thus shown that
every minimal normal subgroup of $G$ has at most two orbits on $\V\Gamma$.
Moreover, if a minimal normal subgroup $N$ has two orbits, then $\Gamma$ is
bipartite with $\{v^N, w^N\}$ being the bipartition of $\Gamma$.

Suppose now that $G$ contains an abelian minimal normal subgroup $N$.
By Lemma~\ref{lem:2orb}, either $N_v\not = 1$ and there exist
two distinct vertices $u,u'\in \V\Gamma$ such that $\Gamma(u) = \Gamma(u')$,
or $N_v= 1$ and $|\V\Gamma| \le 2|N| \le 2^{2k-1}$. The latter case contradicts our
assumption on the order of $\Gamma$, so we may assume that former case happens.
If $k=3$, then Lemma~\ref{lem:unworthy3} yields $\Gamma\cong \K_{3,3}$, while
if $k=4$, then it is easy to see that  $\Gamma\cong\C(r,1)$ 
(see \cite[Lemma 4.3]{Pwilson} for a proof).
We may therefore assume that no minimal normal subgroup of $G$ is abelian.

Suppose now that  a minimal normal subgroup $N$ of $G$ acts semiregularly on $\V\Gamma$.
By Lemma~\ref{lemma:1}, we see that $|N:\cent N g|=1$ or $2$. 
If $|N:\cent N g|=2$, then $N$ is abelian, a contradiction.
Hence $g$ centralises $N$ and
since $g\in G_{vw}$,
we see that $g$ fixes every element in $v^N \cup w^N = \V\Gamma$.
This contradiction shows that none of the minimal normal subgroups of $G$ acts semiregularly on $\V\Gamma$.

Suppose now that $G$ contains two distinct minimal normal subgroups $N$ and $M$.
Let $K_N$ and $K_M$ be the kernels of the actions of $G$ on $\V\Gamma/N$ and $\V\Gamma/M$ respectively.
 Suppose that $N\le K_M$. Then $v^N\subseteq v^{K_M}= v^M$. Let $n\in N$ be an  element of prime order at least $5$. We have $v^n\in v^{M}$ and hence $v^n=v^m$, for some $m\in M$. This gives $nm^{-1}\in G_v$. Since $o(nm^{-1})=\lcm\{o(n),o(m)\}$, it follows that $G_v$ contains an element of order divisible by a prime number at least  $5$. This contradiction shows that $N\nleq K_M$. This yields that $N$ acts faithfully as a group of automorphisms of the graph $\Gamma/M$. However,
 since $M$ is not semiregular, $\Gamma/M$ has valency at most $2$; thus the automorphism group of $\Gamma/M$ is soluble and hence so is $N$. However, this contradicts the fact that $N$ is non-abelian, and thus shows that our initial assumption on the existence of two minimal normal subgroups of $G$ was false. 
\end{proof}

We now continue our analysis under the assumption of Hypothesis~\ref{hyp:sAT}.
Let $N$ be the unique minimal normal subgroup of $G$. Since $N$ is
non-abelian, we see that for some non-abelian simple group $T$ we have:
\begin{equation}
\label{eq:N}
N \cong T_1\times T_2\times\cdots \times T_\ell \>\hbox{ with }\> T_i\cong T\>  \hbox{ for every } i, \>\>\hbox{ and }\>\>
G \lesssim \Aut(T)\wr\Sym(\ell);
\end{equation}
where by $X \lesssim Y$ we indicate that $X$ is a group isomorphic to a subgroup of $Y$.
Observe also that $\cent G N = 1$.
For $h\in G$, let $\sigma_h$ denote the permutation
of $\{1,\ldots,\ell\}$ mapping $i$ to $j$ if and only if $(T_i)^h = T_j$.
Then $\sigma \colon G \to \Sym(\ell)$, $h\mapsto \sigma_h$, is a homomorphism whose kernel equals
\begin{equation}
\label{eq:M}
M:=G\cap \Aut(T)^\ell \le \Aut(T)\wr\Sym(\ell).
\end{equation}
Note that every element $h\in G$ can now be written uniquely as
$(y_1,\ldots,y_\ell)\sigma_h$ for some $y_1, \ldots, y_\ell \in \Aut(T)$.
In particular, let $x_1, \ldots, x_\ell \in \Aut(T)$ be such that
\begin{equation}
\label{eq:g}
g=(x_1,\ldots,x_\ell)\sigma_g.
\end{equation}
Let $K \in \{N, M\}$.
 Since $K\unlhd G$, we see that $K_v\unlhd G_v$. Moreover, since $K_v\not=1$,
 the connectivity of $\Gamma$ implies that
$K_v^{\Gamma(v)}$ is a non-trivial normal subgroup of the $2$-transitive group $G_v^{\Gamma(v)}$. Hence $K_v^{\Gamma(v)}$ is transitive. Since $G_{vw}$ is
the stabiliser of the action of $G_v$ on $\Gamma(v)$, we thus see that $G_v=G_{vw}K_v$. Since $K_v^{\Gamma(v)}$ is transitive,
the quotient $\Gamma/K$ has valence $0$ or $1$ and $|K_v:K_{vw}| = k$. 
In the first case,
$K$ is transitive on $\A\Gamma$, implying that $G=G_{vw}K$, while in the second case, $K$ is edge-transitive and has two orbits on $\A\Gamma$ and $\V\Gamma$, the latter forming the bipartition of $\Gamma$.
In both cases, we see that 
\begin{equation}
\label{eq:Ktr}
K\>  \hbox{ is transitive on }\> \E\Gamma,
\end{equation}
and thus $G=KG_{\{v,w\}}$ with
$|G:KG_v| = |G:KG_{vw}| = 1$ or $2$, depending of whether $\Gamma/K$ has valence $0$ or $1$, respectively. In particular, since $K$ is contained in the
kernel of $\sigma$, this implies that
\begin{equation}
\label{eq:444} 
\sigma(G) = \sigma(G_{\{v,w\}}) \>
\hbox{ and thus }\>
 \sigma(G_{\{v,w\}}) \le \Sym(\ell) \hbox{ is transitive}.
\end{equation}

The structure of the vertex- arc- and edge-stabiliser in a group $G$ acting $2$-arc-transitively on a connected $k$-valent graph with $k\in\{3,4\}$
was first studied by Tutte in his seminar work \cite{tutte} for the case $k=3$, and by 
Weiss~\cite{Weiss} for the case $k=4$. It follows from their work that
$|G_v| \le 48$ if $k=3$ and $|G_v| \le 11\, 664$ if $k=4$.
Furthermore, the triples  $(G_v,G_{vw},G_{\{v,w\}})$ were completely determined (up to
isomorphism of triples of groups) by Conder and Lorimer in \cite{ConLor} for $k=3$,
and by the first-named author of this paper in \cite[Table~1]{Primoz} for $k=4$.
 In Table~\ref{tab:AT2}
we gather some information about these triples that will be frequently used in what follows. In particular, for each of the nine triples, we give
the number of elements of order $2$ and (if $k=4$) of order $3$ in $G_v$.
In the last column, the information on the minimal order of an element $h\in G_{\{v,w\}} \setminus G_{vw}$ is also provided.
\begin{center}
\begin{table}[h!]
\begin{tabular}{||c|c||c|c|c|c||}
\hline\hline
$k$ & type & $|G_v|$ & $|\{ x \in G_v : o(x) = 2\}|$ & $|\{ x \in G_v : o(x) = 3\}|$ & $o(h)$ \\
 \hline\hline
$3$ & $G_5$ &  48 &  19 & & 2 \\ \hline
$3$ & $G_4^1$, $G_4^2$  & 24 & 9 & & 2 \\ \hline
$3$ & $G_3$   & 12  & 7 & & 2 \\ \hline
$3$ & $G_2^2$  & 6 & 3 & & 4  \\ \hline
$3$ & $G_2^1$  & 6 & 3  & & 2 \\ 
 \hline\hline
$4$ & 7-AT  &$11\,664$ & 405 & 890 & 2 \\ \hline
$4$ & 4-AT  &432 & 45 & 80 & 2 \\ \hline
$4$ &  $S_3\times S_4$ & 144 & 39 & 26 & 2 \\ \hline
$4$ &   $C_3\rtimes S_4^*$ & 72 & 21 & 26 & 4 \\ \hline
$4$ &  $C_3\rtimes S_4$ & 72 & 21 & 26 & 2 \\ \hline
$4$ &  $C_3 \times A_4$ & 36 & 3 & 26 & 2 \\ \hline
$4$ &  $S_4$ & 24 & 9 & 8 & 2 \\ \hline
$4$ &  $A_4$x and $A_4$s & 12 & 3 & 8  & 2\\ \hline\hline
\end{tabular}
\medskip
\caption{Vertex-stabilisers of groups $G$ acting $2$-arc-transitively on connected $4$-valent graphs.}
\label{tab:AT2}
\end{table}
\end{center}
\vskip-5mm

With the information provided in Table~\ref{tab:AT2} we can now obtain a series
of useful bounds. For example,
by applying Lemma~\ref{eq:22} with $(G, N,v)$ in place of
$(X,X,\omega)$ we
we see that
\begin{equation}
\label{eq:4444}
|N:\cent N g| \> \le \> |g^G| \> < \>3 |g^G \cap G_v| \>\le\> 
 \left\{ 
 \begin{array}{ll} 3\cdot 19 = 57 & \hbox{ if } k=3 \\ 
    3\cdot 405 = 1215 & \hbox{ if } k=4 \hbox{ and } o(g) = 2\\ 
    3\cdot 890 = 2670 & \hbox{ if } k=4 \hbox{ and } o(g) = 3
  \end{array}
 \right.
\end{equation}
We now split the analysis into two case, depending on whether
$\sigma_g =1$ (or equivalently, $g\in M$) or not.
\medskip

\noindent\textsc{Suppose $\sigma_g\ne 1$.}
\medskip

\noindent 
 Let $\kappa$ be the length of a longest cycle in $\sigma_g$. In particular, $\kappa=o(g)\in \{2,3\}$. Without loss of generality, we may assume that $\sigma_g=(1\,2\cdots\,\kappa)\sigma'$, for some $\sigma'\in \Sym(\{\kappa+1,\ldots,\ell\})$. Since $g^\kappa = 1$, we see that
 $x_1x_2\cdots x_\kappa =1$.
Now consider the element $$h:=(1,x_1^{-1},(x_1x_2)^{-1},\ldots,(x_1x_2\cdots x_{\kappa-1})^{-1},1,1,\ldots,1)\in \Aut(T)^\ell,$$
and observe that
$$h^{-1}gh=(1,\ldots,1,x_{\kappa+1},x_{\kappa+2},\ldots,x_\ell)(1\,2\cdots\,\kappa)\sigma'.$$
Replacing the graph $\Gamma$ with the graph 
$\Gamma^h:=(\V\Gamma,(\E\Gamma)^h)$, the group $G$ with $G^h$ and hence $g$ with $g^h$, we may assume that 
$x_1=x_2=\cdots=x_\kappa=1$. A calculation in $T^\kappa$ gives that $\cent {T^\kappa}{(1\,2\cdots\,\kappa)}$ is the diagonal subgroup $\{(t,\ldots,t)\mid t\in T\}$ of $T^\kappa$. Thus $|\cent {T^\kappa}{(1\,2\cdots\,\kappa)}|=|T|$ and $|\cent N g|\le |T|\cdot |T|^{\ell-\kappa}=|T|^{\ell-\kappa+1}$. Hence
$|N:\cent N g|={|T|^\ell}/{|\cent N g|}\ge {|T|^{\ell}}/{|T|^{\ell-\kappa+1}}=|T|^{\kappa-1}.$
As $|T|\ge 60$, we can now deduce from \eqref{eq:4444}
that 
\begin{equation}
k=4,\> \kappa=o(g) = 2,\> \hbox{ and thus }\>
|N:\cent N g|  = |g^N| \le |g^G| <  1215.
\end{equation}
Assume that $\sigma$ has more than one cycle of length $2$. Without loss of generality we may assume that $\sigma=(1\,2)(3\,4)\sigma''$, for some $\sigma''\in \{5,\ldots,\ell\}$. As above, replacing  $g$ by a suitable $\Aut(T)^\ell$-conjugate, we may assume  $x_3=x_4=1$. A computation gives $|\cent {T^4}{(1\,2)(3\,4)}|=|\{(t,t,t',t')\mid t,t'\in T\}|=|T|^2$ and hence $1215>|N:\cent N g|\ge |T|^2\ge 3600$, which is a contradiction. Thus $\sigma=(1\, 2)$,
\begin{equation}
\label{eq:g11}
g=(1,1,x_3,\ldots,x_\ell)(1\,2)\>\>\>\>\>\hbox{ and }\>\>\>\>\>|N:\cent N g|=|T||T:\cent T {x_3}|\cdots |T:\cent T {x_\ell}|.
\end{equation}

Therefore $|T| \le |N: \cent N g| \le 3\cdot 405 = 1215$, implying
that \begin{equation}
\label{eq:possT}
T\in \{\Alt(5),\Alt(6),\PSL_2(7),\PSL_2(8),\PSL_2(11),\PSL_2(13)\}.
\end{equation}
 Let $V:=\langle g^x\mid x\in G_{\{v,w\}}\rangle$ and observe that $V\le G_{vw}$.
Let  $\Delta$ be the graph defined by $\V\Delta:=\{1,\ldots,\ell\}$ and 
$\E\Delta:=\{ \{r,t\} :  (r\, t) \in \sigma(V)\}$.
Since $(r\, t), (t,s) \in \sigma(V)$ implies $(r\, s) = (r\, t)^{(s\, r\, t)} = 
(r\, t)^{(t\, s)(r\, \,t)} \in \sigma(V)$, we see that every connected component of
$\Delta$ is a complete graph.
Let  $W_1, \ldots, W_k$ be the vertices of the connected components of $\Delta$.
Then for each $i\in \{1,\ldots,k\}$, the group $\sigma(V)$ contains all the transpositions $(r\, t)$ with $r,t\in W_i$, implying that
 $\Sym(W_1) \times \ldots \times \Sym(W_k) \le \sigma(V).$
Now observe that the group $\sigma(G_{\{v,w\}})$
preserves $\E\Delta$ and hence $\sigma(G_{\{v,w\}})$ is a subgroup of $\Aut(\Delta)$,
which by \eqref{eq:444} acts vertex-transitively. In particular, $\Delta$ is vertex-transitive and thus
$|W_i| = m$ for some
 $m \ge 2$ dividing $\ell$ and every $i\in\{1,\ldots,k\}$.
Hence 
\begin{equation}
\label{eq:sig}
\Sym(m)^{\ell/m} \le \sigma(V) \le \sigma(G_{vw}) \le \sigma(G_{\{v,w\}}) \le \Aut(\Delta) = \Sym(m) \wr \Sym(\ell/m).
\end{equation}
Since $|G_{vw}|$ divides $2^2\cdot 3^6$, this implies that
either $m=3$ and $\ell \in \{3, 6\}$ or $m=2$ and $\ell\in \{2,4\}$.


Suppose first that $(m,\ell) = (2,4)$. Since
$\sigma(G_{\{u,v\}})$ is transitive, \eqref{eq:sig} 
implies that $\sigma(G) = \sigma(G_{\{v,w\}}) = \Sym(2)\wr \Sym(2) \cong \D_4$,
and hence $\sigma(V) = \sigma(G_{vw}) = \C_2^2$.
In particular, $|G_{vw}|$ is divisible by $4$,
implying that $G_v$ is of type 7-AT, 4-AT or $S_3\times S_4$.
Moreover,
the kernel $M_{vw}$ of the restriction of $\sigma$ to $G_{vw}$
must be a group of odd order. 
Since $M_v$ is transitive on $\Gamma(v)$,
we see that $|M_v| = 4 |M_{vw}|$.
However, a direct computation
 shows that if $G_v$ is of type 4-AT or 7-AT,
then $G_v$ contains no normal subgroup of order $4$ times an odd integer,
implying that $G_v$ is of type $S_3\times S_4$.
In view of  \eqref{eq:4444}, \eqref{eq:g11} and Table~\ref{tab:AT2}, we see that
$|T|\, |T : \cent T {x_3}|\, |T : \cent T {x_4}| = |N : \cent N g| \le 3\cdot 39$ and thus $T\cong\Alt(5)$ and $x_3=x_4=1$.
Hence $T^4\le G \le \Aut(T)\wr\D_4$, with 
$T = \Alt(5)$, $\sigma(G) = \D_4$, and
$|g^G| < 3\cdot 39$ with $g=(1\, 2)$. We have checked with {\sc Magma} \cite{magma},
that no such group $G$ exists.

Suppose now that $(m,\ell) = (2,2)$.
Then $\sigma(G_{vw}) = \sigma(G_{\{v,w\}}) = \Sym(2)$,
$T^2\le G \le \Aut(T)\wr\Sym(2)$ with 
$T$  as in \eqref{eq:possT}, $\sigma(G) = \Sym(2)$, $g=(1\, 2)$
 and $|g^G| < 1215$.
If $G_v$ is of type 7-AT, then $2^4\cdot 3^6=11\,664=|G_v|$ divides $|G|$, which in turn divides $2 |\Aut(T)|^2$. By inspecting the groups in \eqref{eq:possT},
we see that only  $T=\PSL(2,8)$ satisfies this condition.
A computer assisted computation showed that in this case there are two groups $G$
satisfying the above conditions, however none of the
contains a subgroup isomorphic to the vertex-stabiliser of type 7-AT.
Hence $G_v$ is not of type 7-AT. But then, in view of  \eqref{eq:4444} and Table~\ref{tab:AT2}, we have $|g^G| < 3\cdot 45$. Checking the groups in \eqref{eq:possT} and all the groups $G$ satisfying $T^2\le G \le \Aut(T)\wr\Sym(2)$ and  $\sigma(G) = \Sym(2)$,
 we see that $|g^G| < 3\cdot 45$ holds only when $T=\Alt(5)$
 with $|g^G| = 60$ or $120$, implying that $G_v$ is of type $C_2\rtimes S_4$, 
 $C_2\rtimes S_4^*$, $S_3\times S_4$ or 4-AT. In particular,
 $|G_v| \ge 72$, and since $|G|\le 2|\Sym(5)|^2 = 28\,000$, we see that
 $|\V\Gamma| \le 400$. However, all $2$-arc-transitive graphs of order at 
 most $512$ are known (see \cite{Primoz}) and it can be easily checked
 that none of these graphs, with the exception of $\Psi_1, \ldots, \Psi_6$ and
 $\C(4,s)$ with $s\in\{1,2\}$, has a non-trivial automorphism fixing more than $1/3$ of the vertices.
 
Suppose now that $(m,\ell) = (3,3)$. Then 
\begin{equation}
\label{eq:33}
\sigma(G) = \sigma(G_{\{v,w\}}) = \Sym(3),\> T^3\le G \le \Aut(T)\wr\Sym(3), \>
T \hbox{ as in  \eqref{eq:possT}, and }  g=(1,1,x_3)(1\, 2).
\end{equation}
 If $x_3\not =1$, then in view of \eqref{eq:g11}, we have
 $60\le |T| < |G_v| / |T: \cent T {x_3}|$.
 By inspecting the centralisers of involutions of the simple group in \eqref{eq:possT}, we see that  $T=\Alt(5)$ and $G_v$ is of type 7-AT.
 However, $|\Aut(T) \wr \Sym(3)|$ is not divisible by $|G_v|=11\, 664$ in this case, yielding a contradiction. Hence $x_3=1$ and thus $g=(1\, 2)$.
If $G_v$ is of type 7-AT, then the divisibility condition $|G_v| \mid |\Aut(T) \wr \Sym(3)|$ yields $T\in \{\Alt(6), \PSL(2,8)\}$.
If $T=\Alt(6)$, then no group $G$ satisfying \eqref{eq:33} is such that  $|g^G| \le 1215$. If $T=\PSL(2,8)$, then there are 25 groups $G$ satisfying \eqref{eq:33},
with the minimum value of $|g^G|$ being $1080$. Now observe that
$g$ is not a square of any element in $\Aut(T) \wr \Sym(3)$.
A direct inspection of the vertex-stabiliser of type 7-AT reveals that there
are only $324$ involutions in $G_v$ that are non-squares, implying that
$|g^G \cap G_v| \le 3\cdot 324$, which contradicts the fact that $1080 \le |g^G| \le |g^G \cap G_v|$. Hence $G_v$ is not of type 7-AT.
 By \eqref{eq:4444} and Table~\ref{tab:AT2}, it follows that 
 $|g^G| \le 3\cdot 45$ and
 $|T| = |N:\cent N g| \le 3 | g^G \cap G_v| \le 3\cdot 45$, forcing $T=\Alt(5)$. 
 However, direct computation shows
 that no group $G$ satisfying \eqref{eq:33} such that $|g^G| \le 3\cdot 45$ exists in this case.

Suppose finally  that $(m,\ell) = (3,6)$. Then $\sigma(G_{uv})$ contains a subgroup isomorphic
to $\Sym(3)\times\Sym(3)$. Inspecting the orders of the arc-stabilisers in Table~\ref{tab:AT2},
we see that $(G_v,G_{vw},G_{\{v,w\}})$ is of type 7-AT, 4-AT or $S_3\times S_4$
and that $\sigma(V) = \sigma(G_{vw}) =\Sym(3)\times\Sym(3)$.
Similarly as in the case $(m,\ell) = (2,4)$, we see 
that $M_{vw}$ has odd order and thus $G_v$ contains a subgroup
of order $4$ times an odd number, which rules out the types 4-AT and 7-AT.
But then, in view of \eqref{eq:g11}, we see that
$|N : \cent N g| = |T| |T:\cent T {x_3}|\cdots |T:\cent T {x_6}|  \le 3\cdot 39$, implying that
 $T\cong\Alt(5)$ and $x_3 = \ldots = x_6 = 1$.
 Now let $h$ be an element of minimal order in $G_{\{v,w\}} \setminus G_{vw}$.
 According to Table~\ref{tab:AT2}, we see that $o(h) = 2$.
 Consider the group $L:=\langle M, g, h\rangle$. Since $M$ is transitive on $\E\Gamma$
 (see \eqref{eq:Ktr})  and since $h$ swaps the arc $(v,w)$, we see that $L$ is an arc-transitive subgroup
 of $G$ containing $g$. By Theorem~\ref{the:GHAT}, $L$ is $2$-arc-transitive, and by Hypothesis~\ref{hyp:sAT}, it follows that $G=L$.
 Now,  since 
 $G_{\{v,w\}} = G_{vw}\langle h \rangle$ and since
 $\sigma(G_{\{v,w\}})$ is transitive on $\{1,\ldots, 6\} = \V\Delta$, we see that
 $\sigma(h)$ swaps the two connected component $W_1$ and $W_2$ of $\Delta$.
  By construction,
 one connected component of $\Delta$ contains the vertices $1$ and $2$, and
 without loss of generality, we may assume that $W_1= \{1,2,3\}$ and $W_2=\{4,5,6\}$
 and hence that $\sigma(h)=(1\, 4)(2\, 5)(3\, 6)$. But then we see
 that $\langle T_1, T_2, T_4, T_5\rangle$ is normalised by $M,g$ and $h$ and thus by 
 $G = \langle M,g,h\rangle$, which contradict the assumption that $N$ is a minimal normal 
 subgroup of $G$.
%
\smallskip

\noindent\textsc{Suppose $\sigma_g= 1$.}
\medskip

\noindent Then $g=(x_1,x_2,\ldots,x_\ell) \in M$, where $M$ is as in
\eqref{eq:M}.
Let $h$ be an element of $G_{\{v,w\}}\setminus G_{vw}$ of minimal possible order. 
From the information given in Table~\ref{tab:AT2}, it follows that
 $o(h)\in \{2,4\}$; moreover, $o(h)=4$ if and only if $k=4$ and $(G_v,G_{vw},G_{\{v,w\}})$ is of type $C_3\rtimes S_4^*$, or $k=3$ and $(G_v,G_{vw},G_{\{v,w\}})$ is of type $G_2^2$,
Now observe that $\langle M,h\rangle = M\langle h\rangle \le G$  acts arc-transitively on $\Gamma$. Since $G$ is a smallest arc-transitive group of $\Gamma$ containing the element $g$, it follows that $G=M\langle h \rangle$. Since $M$
is the kernel of the homomorphism $\sigma\colon G \to \Sym(\ell)$,
we see that
$\langle h \rangle = \sigma(\langle h \rangle) = \sigma(G) = \sigma(G_{\{u,v\}})$, which is by \eqref{eq:444} a transitive subgroup of $\{1,\ldots,\ell\}$. Hence $\ell\in \{1,2,4\}$. 

For a finite simple group $X$, embedded as the group of inner automorphisms into $\Aut(X)$, and an integer $r\ge 2$ such that $X$, (respectively, $\Aut(X)$) contains an element of order $r$, let 
\begin{eqnarray*}
\iota(X,r) &:= &\min \{|X: \cent X x| : x \in X, o(x)=r\};\\
\iota_*(X,r) & := &\min \{|X: \cent X x| : x \in \Aut(X), o(x)=r\}; \\
m(X) & := & \min\{|X:H| : H\le X, H\not = X\}.
\end{eqnarray*}
Note that $m(X)\le \iota_*(X,r) \le \iota(X,r)$ and that $m(X)$ equals the minimal
degree of a faithful transitive permutation representation of $X$.

Now observe that
 $|N : \cent N g| = |T : \cent T {x_1}|\, |T : \cent T {x_2}| \ldots |T : \cent T {x_\ell}|$.
 Let $\alpha:=\{ i\in \{1,\ldots,\ell\} : x_i\not =1\}$ and observe that $\alpha\ge 1$.
 Inequality \eqref{eq:4444} and Table~\ref{tab:AT2} now
imply that 
 \begin{equation}
 \label{eq:mT}
  m(T)^\alpha \le \iota_*(T,o(g))^\alpha < |g^G \cap G_v| \le \left\{
    \begin{array}{ll} 
     57; & \hbox{ if } k=3;\\ 
    1215; & \hbox{ if } k=4 \hbox{ and } o(g)=2 ;\\ 
    2670; & \hbox{ if } k=4 \hbox{ and }o(g)=3.
   \end{array}
    \right.
 \end{equation}

The values of $m(T)$ for  finite simple groups $T$ are known
and can be found, for example, in~\cite[Table~4]{GMPS} for the groups of Lie type (this table takes in account the corresponding table in~\cite[Table~$5.2A$]{KL} together with the corrections of Mazurov and Vasil'ev in~\cite{33}) and  in \cite{wilsonArXiv} or \cite{ATLAS}, for sporadic groups. In Table~\ref{table:1}, containing all non-abelian simple groups $T$ with $m(T)<2670$, we summarise the relevant information; note that the last two columns 
give a condition for the group in the corresponding row satisfies $m(T)<2670$ and $m(T)<117$,
respectively (the meaning of the bound $117$ will become apparent later).
\begin{table}[!ht]
\begin{tabular}{|c|c|c|c|c|c|}\hline
Group $T$& $m(T)$& $m(T) < 2670$ & $m(T) < 117$ \\
\hline\hline
$\Alt(n)$& $n$  &$5\le n\le 2669$ & $5\le n\le 116$  \\
\hline
$\PSL_2(q)$& $q+1$ &$8\le q\le 2663$, $q\neq 9,11$ & $8\le q\le 113$, $q\not = 9,11$ \\
$\PSL_3(q)$& $(q^3-1)/(q-1)$ & $2\le q\leq 49$ & $2\le q\le 9$ \\
$\PSL_4(q)$& $(q^4-1)/(q-1)$     &$3\le q\le 13$ &  $3\le q\le 4$\\
$\PSL_5(q)$&  $(q^5-1)/(q-1)$    &$2\le q\le 5$  &  $q=2$ \\
$\PSL_d(2)$&  $2^d-1$     &$6\le d\le 11$ & $d=6$  \\
$\PSL_2(11)$ & $11$ & true & true \\
$\PSL_6(3)$, $\PSL_7(3)$, $\PSL_6(4)$ & $364, 1093, 1365$ &  true & false \\
 \hline

$\PSp_4(q)$& $(q^4-1)/(q-1)$     &$4\le q\le 13$ & $q=4$ \\
$\PSp_{2m}(2)$& $2^{m-1}(2^m-1)$        &  $3\le m\le 6$ &$m =3$ \\
$\PSp_6(3), \PSp_6(4)$ & $364,1365$ & true & false \\
\hline

$\PSU_3(q)$ &  $q^3+1 $   & $3\le q \le 13$, $q\neq 5$ & $3\le q \le 4$ \\
$\PSU_4(q)$ & $(q+1)(q^3+1)$     &$2\le q\le 5$ & $2\le q \le 3$ \\
$\PSU_3(5)$   & $50$ & true & true \\
$\PSU_5(2)$, $\PSU_6(2)$, $\PSU_5(3)$    & $165, 672, 2440$ & true & false \\
\hline

$\mathrm{P}\Omega_{2m}^{+}(2)$,  & $2^{m-1}(2^m-1)$  & $4\le m \le 6$  & false\\
$\mathrm{P}\Omega_{2m}^{-}(2)$&   $2^m(2^{m-1}-1)$      &$4\le m \le 6$ & $m=4$\\
$\mathrm{P}\Omega_7(3)$, $\mathrm{P}\Omega_8^{-}(3)$, $\mathrm{P}\Omega_8^{+}(3)$ &   $351, 1066, 1080$    & true & false \\
\hline

$G_2(3)$, $G_2(4)$&  $351,416$   & true & false\\
$^{2}B_2(8)$, &  $65$      & true & true \\
$^{3}D_4(2)$, $^{2}B_2(32)$, $^{2}F_4(2)'$ &  $819, 1025, 1755$    & true & false \\
\hline

${\rm M}_{11}$, ${\rm M}_{12}$, ${\rm M}_{22}$, ${\rm M}_{23}$, ${\rm M}_{24}$  & $11,12,22,23,24$      & true  & true \\
${\rm J}_1$,  ${\rm McL}$& $266,275$       & true & false \\
$ {\rm J}_2$, ${\rm HS}$& $100,100$       & true & true  \\
${\rm Co}_2$, ${\rm Co}_3$, ${\rm Suz}$, ${\rm He}$& $2300,276,1782,2058$       & true & false \\
\hline
\end{tabular}
\medskip
\caption{Simple groups $T$ with $m(T)< 2670$.}
\label{table:1}
\end{table}
%
We will now consider the
 possible values of $\ell$ case by case and show that cases $\ell = 4$ and $\ell=2$ lead
 to a contradiction.
 \smallskip

Suppose first that $\ell =4$.
Recall that in this case $o(h) = 4$ and $(G_v,G_{vw},G_{\{v,w\}})$ is of type $C_3\rtimes S_4^*$
if $k=4$ or of type $G_2^2$ if $k=3$. If $k=3$, then $m(T) \le \iota_*(T,2) \le 8$,
implying that $T$ embeds into $\Sym(n)$ for some $n\in \{5,6,7,8\}$.
But then $T$ embeds into $\Sym(m)$ for $m\le 8$. 
Hence either $T=\Alt(n)$ for $n\in \{5,\ldots,8\}$ or
$T=\PSL(3,2)$. However, a closer inspection of these groups shows that none of them satisfies
$\iota_*(T,3) \le 8$.
We may therefore assume that $k=4$ and that $(G_v,G_{vw},G_{\{v,w\}})$ is of type $C_3\rtimes S_4^*$.
From the information given in \cite[Table 1]{Primoz}, we see that
$|G_v|=2^3\cdot 3^2$, $|G_{vw}|=2\cdot 3^2$, $|G_{\{v,w\}}|=2^2\cdot 3^2$,
the Sylow $3$-subgroup $P$ of $G_{vw}$ is normal in $G_{vw}$, $P\cong \C_3^2$, and 
 $h^2$ inverts every element of $P$.
Since $G=MG_{\{v,w\}}$, we see that $G/M \cong G_{\{v,w\}}/(M \cap G_{\{v,w\}})
= G_{\{v,w\}}/M_{\{v,w\}}$. 
Without loss of generality, $h=(y_1,y_2,y_3,y_4)(1\,2\,3\,4)$ for some $y_i\in \Aut(T)$, implying that $G/M \cong \C_4$. But then $|M_{\{v,w\}}| = 3^2$ and
since $|M_{\{v,w\}}:M_{vw}| \le 2$, we see that
$M_{\{v,w\}} = M_{vw}$ and thus $M_{vw}=P$; in particular, $o(g) = 3$ and $g^{h^2} = g^{-1}$. 
 Now,  $h^2=(y_4y_1,y_1y_2,y_2y_3,y_3y_4)(1\,3)(2\,4)$, and thus
$ (x_1^{-1}, x_2^{-1},x_3^{-1},x_4^{-1}) =    (x_3^{y_2y_3},x_4^{y_3y_4},x_1^{ y_4y_1}, x_2^{y_1y_2})$.
 Since $g\not = 1$, this implies that at least two of the elements
 $x_1,\ldots,x_4$ are non-trivial. In view of \eqref{eq:mT}
  we see that
 $m(T) \le \iota_*(T,3)^2 < 3\cdot 26 \> =\> 78$, and hence $\iota_*(T,3) \le 8$.
However, as we have shown in case $k=3$, no simple group $T$ satisfies this condition.
 This contradiction shows that $\ell \not = 4$.
\smallskip

Suppose now that $\ell =2$.
%
%
Then $h=(y_1,y_2)(1\,2)$ for some $y_1, y_2\in \Aut(T)$,
implying that $G_{\{v,w\}}/M_{\{v,w\}}\cong G/M\cong \C_2$.
Since $g\in M$, the minimality of $G$ then implies that $M$ is not arc-transitive,
showing that $\Gamma$ is bipartite with $\{v^M,w^M\}$ being the bipartition,
and that $|M_v| = |M|/ |v^M| = |G|/|\V\Gamma| = |G_v|$. In particular, $M_v = G_v$ and
$M$ is the kernel of the action of $G$ on the bipartition.
Consider the groups $L_1:=M \cap ( \Aut(T_1)\times \{1\})$ and
$L_2:=M \cap (\{1\} \times \Aut(T_2))$.
Note that both $L_1$ and $L_2$ are normal in $M$, that $L_1\cap L_2 =1$ ,
and that conjugation by $h$ swaps $L_1$ with $L_2$. 
Hence $L:=\langle L_1, L_2 \rangle \cong L_1\times L_2$ is a normal subgroup of $G=M\langle h \rangle$. Moreover, since $T_1 \times T_2=N\le M$, we see that $T_i \le L_i$ for $i\in\{1,2\}$.

Suppose that $g$ is contained in one of the group $L_1$ or $L_2$. Without loss
of generality, we may assume that $g\in L_1$, and thus $(L_1)_{vw} \not =1$.
Since $L_1$ is normal in $M$ and since
$v^M \cup w^M = \V\Gamma$, we see that
 $(L_1)_u\not =1$ for every $u\in \V\Gamma$.
The connectivity of $\Gamma$ then implies that $(L_1)_u^{\Gamma(u)}\not =1$, and since $G_u^{\Gamma(u)}$ is primitive, we see that $(L_1)_u^{\Gamma(u)}$ is transitive
for every $u\in \V\Gamma$.
Hence $\Gamma/L_1 \cong \K_2$, implying that $v^{L_1} = v^M$.
Therefore $M=L_1M_v$ and thus $M/L_1  \cong L_1M_v / L_1 \cong M_v/(L_1)_v$.
Since $M_v$ is soluble, so is $M/L_1$; however, $M/L_1$ contains a subgroup isomorphic
to $L_2$, which is non-soluble since it contains $T_2$.

This contradiction shows that $g$ is contained neither in $L_1$ nor in $L_2$ and thus
$g=(x_1,x_2)$ with both $x_1$ and $x_2$ nontrivial. In view of inequality \eqref{eq:mT}
(where we may assume $\alpha \ge 2$)
and Table~\ref{tab:AT2},
we thus see that 
 \begin{eqnarray*}
\iota_*(T,2) \le 7 &\hbox {if}& k=3\\
\iota_*(T,2) \le 34\> \hbox{ or }\> \iota_*(T,3) \le 51 &\hbox {if}& k=4 \hbox{ and } G_v \hbox{ is of type 7-AT,}\\
\iota_*(T,2) \le 11\> \hbox{ or }\> \iota_*(T,3) \le 15 &\hbox {if}& k=4 \hbox{ and } G_v \hbox{ is not of type 7-AT.}
 \end{eqnarray*}

We have already seen that no non-abelian simple group $T$ satisfies $\iota_*(T,2) \le 7$.
We may thus assume that  $k=4$.
If $G_v$ is not of type 7-AT, then one can eaily use a computer algebra system, such as {\sc Magma} \cite{magma}, to check that none of the groups $T$ in Table~\ref{table:1} with $m(T)\le 15$
satisfies the  second of the above conditions. Similarly, if $G_v$ is of type 7-AT,
then $|G_v| = 11664$ and since $G\le \Aut(T) \wr \Sym(T)$, we see that
$11664$ divides $2 |\Aut(T)|^2$. By first checking the groups $T$  in Table~\ref{table:1} with  $m(T) \le 51$ against this divisibility condition and then, for the remaining groups, 
directly computing the values $\iota_*(T,r)$, $r\in\{2,3\}$, one sees that no groups $T$ satisfying the first of the above conditions exists either. This shows that $\ell \not = 2$.
\smallskip

We may thus assume for the rest of the proof that $\ell = 1$; that is $T \le G\le \Aut(T)$
where  $T$ is the unique minimal normal subgroup of $G$ and $\cent G T =1$.
If $T_v=1$, then Lemma~\ref{lemma:1} implies that $|T : \cent T g| = 1/\fpr_{v^T}(g) < 3$,
implying that $g$ centralises $T$, contradicting the fact that $\cent G T = 1$.
Since $T$ is normal in $G$, we thus see that  $T_v$ is transitive on
$\Gamma(v)$ and $\Gamma/T \cong \K_2$ or $\K_1$. 
We will now split our analysis depending on the valence of $\Gamma$.
\smallskip

Suppose first that $k=4$.
Let $H:=\langle T, h\rangle = T\langle h \rangle$ and observe that $H$ is arc-transitive.
Moreover, $T = H$ (which happens if $\Gamma/T \cong \K_1$) or
$T$ has index $2$ in $H$ (which happens if $\Gamma/T \cong \K_2$).
In both cases we have $T_v = H_v$, implying that $T_v$ is isomorphic to
one of the nine  possible vertex-stabilisers 
of $4$-valent, $2$-arc-transitive graphs given
in Table~\ref{tab:AT2}. 
Now, observe that the vertex-stabiliser of type 4-AT or 7-AT contains
no proper normal subgroup isomorphic to one of the stabilisers
in Table~\ref{tab:AT2}. This implies that either $T_v = G_v$ (and thus $g\in T_v$ 
and $|G:T| \le 2$) or $G_v$ is not of type 4-AT or 7-AT.
Having in mind that $g\in T$ implies that the expression $\iota_*(T,o(g))$ in
\eqref{eq:mT} can be substituted with $\iota(T,o(g))$ and using the information
from Table~\ref{tab:AT2}, we can now conclude that one of the following holds
(here part (b) corresponds to
the case when $G_v$ is of type 4-AT and part (c) to the case when 
$G_v$ is of type 7-AT): 
\begin{itemize}
\setlength{\itemsep}{1pt}
\item[{\rm (a)}]
 $\iota_*(T,2) < 117$ or $\iota_*(T,3) < 78$;
\item[{\rm (b)}]
$|T|$ is divisible by $432$,
 and $\iota(T,2) < 135$ or $\iota(T,3) < 240$;
\item[{\rm (c)}]
$|T|$ is divisible by $11664$, 
 and  $\iota(T,2) < 1215$ or $\iota(T,3) < 2670$.
 \end{itemize}

Non-abelian simple groups $T$  satisfying one of the above conditions
can now be determined using purely theoretical argument or in combination
with computer assisted computations.
For example, for the alternating groups $\Alt(n)$, $n\ge 5$, it is well-known and easy to see that:
$$
\iota(\Alt(n),3) = \iota_*(\Alt(n),3) = 2{n \choose 3},\>
\iota_*(\Alt(n),2) = {n \choose 2},\>
\iota(\Alt(n),2) = 3{n \choose 4} \hbox{ for } n\neq 8, \iota(\Alt(8),2) = 105.
$$
From this we see that $\Alt(n)$ satisfies
(a) if and only if $5\le n\le 15$, that it never satisfies (b), and
that it satisfies (c) if and only if $15\le n \le 16$.
To determine the non-alternating groups $T$ satisfying (a),
we have considered all the groups $T$ in Table~\ref{table:1} satisfying $m(T)<117$
 (see the last column of the table),
and then compute the values $\iota_*(T,2)$ and $\iota_*(T,3)$ directly with {\sc Magma}.
%
The groups $T$ satisfying conditions (b) and (c)
were determined by first checking divisibility conditions on $|T|$ and then checking the bounds on $\iota(T,r)$ directly with {\sc Magma}. 
%
%
This computations resulted in the following list of groups $T$ satisfying at least one of the conditions (a), (b) and (c):
\begin{eqnarray*}
& & \Alt(n) \hbox { with } 5\le n \le 16,\,
 \PSL_2(8), \PSL_2(11), \PSL_2(13), \PSL_2(16), \PSL_2(25), \PSL_3(2), \PSL_3(3),\PSL_4(3),\\
& &  \PSU_3(3),\PSU_4(2),  \PSU_4(3), \PSU_5(2),  \PSU_6(2),
\PSp_6(2), \PSp_6(3), \PSp_{10}(2), \POmega_7(3), G_2(3)
\end{eqnarray*}
%
%
%
To deal with these possible groups $T$ and corresponding groups $G$ with
$T\le G \le \Aut(T)$, consider a chain
$$
G_v := X_1 < X_2 < \ldots < X_{m-1} < X_m := G
$$
such that each $X_i$, $i\in\{1,\ldots, m-1\}$, is a maximal subgroup of $X_{i+1}$
Let $k$ be the smallest index such that $T\le X_k$. Then, for each $i\in \{1,\ldots, k-1\}$,
 the action of $G$ by right multiplication on the cosets of $X_i$ in $G$ is faithful
 and in view of Lemma~\ref{lem:quofix}, we see that $g$ is a non-trivial permutation
 of $X_i\backslash G$ with $\fpr_{X_i\backslash G}(g) > 1/3$.
 This observation allows us to use the following naive algorithm 
 which finishes the proof of Theorem~\ref{thrm:1}.
 
Let $T$ be one of the groups satisfying a condition (a), (b) or (c) and let $G$ be such that $T\le G\le \Aut(G)$.
Initialise the procedure by letting $\mathcal{Y}:=\{G\}$.
 Now construct a set $\mathcal{Z}$ by going through all
the group $Y\in \mathcal{Y}$ and then through all the maximal subgroups $M$ of $Y$ (modulo conjugation in $Y$).
Put $M$ into $\mathcal{Z}$ if and only if either $T\le M$ or
there exists an element $g\in G$ with $o(g)\in\{2,3\}$ such that
$\fpr_{M\backslash G}(g) > 1/3$  (this can be checked by determining
the  set $M \cap g^G$ of elements in $M$ that are conjugate in $G$ to $g$,
dividing its size by $|g^G| = |G : \cent G g|$, and checking if ratio is larger than $1/3$).
In the latter case,
check it $M$ is isomorphic to a possible vertex-stabiliser of a 
connected $4$-valent $2$-arc-transitive graph in~\cite[Table~1]{Primoz},
and if it is,
check if any of the orbital graphs of $G$ acting on $M\backslash G$
is a connected $4$-valent graph with $G$ acting
$2$-arc-transitively on it. If there is such a graph, store it.
Finally, we repeat this procedure with 
$\mathcal{Z}$ in place of $\mathcal{Y}$, until the set $\mathcal{Y}$ becomes empty.

This computation might seem very time and memory consuming but for most groups $T$ the
procedure stops after the first few iteration. The resulting graphs
are: $\Psi_1,\Psi_2,\Psi_3,\Psi_4, \Psi_5$, arising from
$T=\Alt(5), \Alt(5)$, $\PSL(3,2)$, $\PSL(3,3)$ and  $\Alt(7)$, respectively.
This finishes the case $k=4$ and thus proves Theorem~\ref{thrm:1}.
\medskip

Let us now assume that $k=3$. By \eqref{eq:mT} we see that $m(T) \le \iota_*(T,2)\le 56$,
and if $G_v$ is not of type $G_5$, then we obtain that $m(T) \le \iota_*(T,2)\le 27$.
Similarly as in the case $k=4$, a computer assisted inspection of the groups in Table~\ref{table:1} 
yields that the only non-abelian simple groups $T$ satisfying $\iota_*(T,2)\le 26$ are 
$\Alt(5), \Alt(6),\Alt(7)$ and $\PSL(3,2)$. Since $|\Aut(T)|/6 \le 840$, we see that
all graphs arising from a $2$-arc-transitive action of $G$ have order at most $840$,
contradicting our assumption that $|\V\Gamma| > 10\,000$. We may thus assume that
$G_v$ is of type $G_5$, and thus that $|\Aut(T)|$ is divisible by $48$ and that
 $|\Aut(T)|/48 > 10\,000$. The only group from Table~\ref{table:1} satisfying these
 restrictions together with $\iota_*(T,2)\le 56$ is $\Alt(11)$. Using the algorithm described
 at the end of the case $k=4$ reveals that no graph satisfying Hypothesis~\ref{hyp:sAT} arises
 in this case. This completes the proof of Theorem~\ref{thrm:2}.

\end{document}